\documentclass{article}

\PassOptionsToPackage{numbers, compress}{natbib}

\usepackage[final]{neurips_2025}

\usepackage[utf8]{inputenc} 
\usepackage[T1]{fontenc}    
\usepackage{hyperref}       
\usepackage{url}            
\usepackage{booktabs}       
\usepackage{amsfonts}       
\usepackage{nicefrac}       
\usepackage{microtype}      
\usepackage{xcolor}         

\usepackage{amssymb,amsmath,amsthm}
\usepackage{mathtools}
\usepackage{stmaryrd}
\usepackage{bbm,bm}

\usepackage{algorithm}
\usepackage{algorithmicx}
\usepackage{algpseudocode}

\usepackage{thmtools}
\declaretheorem{theorem}

\declaretheorem{lemma}
\declaretheorem{proposition}

\declaretheorem{fact}
\declaretheorem{claim}

\declaretheoremstyle[qed=$\square$]{definitionwithend}
\declaretheorem{definition}
\declaretheorem{assumption}
\declaretheorem[style=definitionwithend]{example}
\declaretheorem{remark}

\usepackage{graphicx}
\usepackage{subfigure}
\usepackage{parskip}
\usepackage{booktabs}
\usepackage{url}
\usepackage{hyperref}
\hypersetup{
	colorlinks=true,
	linkcolor=blue,
	filecolor=blue,
	anchorcolor=blue,
	urlcolor=cyan,
	citecolor=purple
}
\usepackage{enumerate}
\usepackage[shortlabels]{enumitem}
\usepackage{verbatim}
\usepackage{booktabs}       
\usepackage{multirow}

\usepackage{url}            

\let\emptyset\varnothing

\usepackage{mathrsfs}

\def\B{{\mathbb{B}}}

\def\N{{\mathbb{N}}}

\def\P{{\mathbb{P}}}

\def\R{{\mathbb{R}}}

\def\bA{{\mathbf{A}}}

\def\cA{{\cal A}}

\def\cD{{\cal D}}

\def\cO{{\cal O}}

\def\cS{{\cal S}}

\def\cV{{\cal V}}

\def\cX{{\cal X}}
\def\cY{{\cal Y}}

\def\s{{\bm s}}

\def\u{{\bm u}}
\def\v{{\bm v}}

\def\x{{\bm x}}
\def\y{{\bm y}}
\def\z{{\bm z}}
\def\bz{{\bm 0}}

\def\1{{\bm 1}}

\newcommand{\dist}{\operatorname{dist}}

\usepackage{mdframed} 

\DeclareMathOperator*{\argmin}{arg\,min}

\def\prox{{\operatorname{prox}}}

\DeclareMathOperator*{\E}{\mathbb{E}}

\DeclareMathOperator{\dom}{dom}

\usepackage{soul}

\title{
	Set Smoothness Unlocks Clarke Hyper-stationarity\\  in Bilevel Optimization
}

\author{%
	He Chen\\
	SEEM\\
	The Chinese University of Hong Kong\\
	Shatin, Hong Kong \\
	\texttt{hchen@se.cuhk.edu.hk} \\
	\And
	Jiajin Li\\
	Sauder School of Business\\
	University of British Columbia\\
	Vancouver, BC, Canada\\
	\texttt{jiajin.li@sauder.ubc.ca} \\
	\AND
	Anthony Man-Cho So\\
	SEEM\\
	The Chinese University of Hong Kong\\
	Shatin, Hong Kong \\
	\texttt{manchoso@se.cuhk.edu.hk} \\
}

\begin{document}

	\maketitle

	\begin{abstract}
		Solving bilevel optimization (BLO) problems to global optimality is generally intractable. A common surrogate is to compute a hyper-stationary point—a stationary point of the hyper-objective function obtained by minimizing or maximizing the upper-level objective over the lower-level solution set.
		Existing methods, however, either provide weak notions of stationarity or require restrictive assumptions to guarantee the smoothness of hyper-objective functions. In this paper, we eliminate these impractical assumptions and show that strong (Clarke) hyper-stationarity remains computable even when the hyper-objective is nonsmooth.
		Our key ingredient is a new structural property, called \emph{set smoothness}, which captures the variational dependence of the lower-level solution set on the upper-level variable. We prove that this property holds for a broad class of BLO problems and ensures weak convexity (resp. concavity) of pessimistic (resp. optimistic) hyper-objective functions. 
		Building on this foundation, we show that a zeroth-order algorithm that computes approximate Clarke hyper-stationary points with non-asymptotic convergence guarantees. To the best of our knowledge, this is the first computational guarantee for Clarke-type stationarity in nonsmooth BLO.  Beyond this specific application, the set smoothness property emerges as a structural concept of independent interest, with potential to inform the analysis of broader classes of optimization and variational problems.

	\end{abstract}

	\section{Introduction}
	Bilevel optimization (BLO) models hierarchical decision-making with two agents acting sequentially \cite{dempe2002foundations,dempe2020bilevel}. The follower responds to the leader’s decision by solving a lower-level optimization problem, while the leader seeks an optimal strategy to minimize its upper-level objective subject to this reaction. The follower’s attitude plays a central role: If the follower is favorable (resp. adverse) to the leader, the resulting BLO is termed optimistic (resp. pessimistic) \cite{dempe2002foundations,wiesemann2013pessimistic,liu2018pessimistic}. Formally, the optimistic and pessimistic BLO take the following forms:
	
	\begin{minipage}{0.5\textwidth}
		\vspace{0.5em}
		\quad\textit{Optimistic BLO:}
		\begin{equation*}
			\begin{array}{cl}
				\min \limits_{\x\in\R^m}\min\limits_{\y\in\R^n}&\  F(\x,\y)\\ \text { subject to } 
				& \y\in\underset{\y^{\prime}\in\R^n}{\argmin}\ f(\x,\y^{\prime}), \\
			\end{array}   
		\end{equation*}
		\vspace{0.5em}
	\end{minipage}%
	\Bigg|
	\begin{minipage}{0.5\textwidth}
		\vspace{0.5em}
		\quad\textit{Pessimistic BLO:}
		\begin{equation*}\label{eq:BLO}
			\begin{array}{cl}
				\min \limits_{\x\in\R^m}\ \max\limits_{\y\in\R^n}&\  F(\x,\y) \\ \text { subject to } 
				&\y\in\underset{\y^{\prime}\in\R^n}{\argmin}\ f(\x,\y^{\prime}). \\
			\end{array}   
		\end{equation*}
		\vspace{0.5em}
	\end{minipage}
	These formulations appear in diverse domains such as Stackelberg games \cite{dempe2020bilevel,bruckner2011stackelberg,wang2021fast}, hyperparameter optimization \cite{franceschi2017forward,bertrand2020implicit,chen2024lower}, reinforcement learning \cite{konda2003onactor,zeng2024two,hong2023two}, and interdiction games \cite{liu2018pessimistic,caprara2016bilevel}, among others. A standard approach to tackle such nested problems is to reformulate them into single-level problems via \emph{hyper-objective functions}.
Let $\cS(\x)\coloneqq{\argmin}_{\y^{\prime}\in\R^n}\ f(\x,\y^{\prime})$ denote the follower’s optimal response set.
The \emph{optimistic} and \emph{pessimistic} hyper-objectives are then defined as  
	\begin{equation}\label{eq:varphi}
		\varphi_{o}(\x)\coloneqq \min\limits_{\y\in \cS(\x)}\ F(\x,\y),
\qquad
		\varphi_{p}(\x)\coloneqq \max\limits_{\y\in \cS(\x)}\ F(\x,\y).
	\end{equation}
Solving an optimistic (resp. pessimistic) BLO is therefore equivalent to minimizing the corresponding hyper-objective $\varphi_o$ (resp. $\varphi_p$). 

Despite the single-level reformulation, the resulting hyper-objective functions are highly nonconvex \cite{chen2024finding,kwon2023penalty}, which makes global optimization intractable.
In practice, researchers therefore focus on finding stationary points rather than global minimizers, using algorithms such as implicit gradient descent \cite{franceschi2017forward,franceschi2018bilevel,bertrand2020implicit} and fully first-order methods \cite{kwon2023fully,hong2023two}.  These approaches assume that the lower-level problem is strongly convex, ensuring a unique solution, i.e., $\cS(\x)\coloneqq\{\y^\star(\x)\}$. Under this assumption, the hyper-objective reduces to a smooth function $\varphi(\x) \coloneqq F(\x, \y^\star(\x))$ \cite{ghadimi2018approximation}. One can then seek an $\epsilon$-approximate hyper-stationary point satisfying $\|\nabla \varphi(\x)\| \leq \epsilon$. Convergence is well understood under smoothness and uniqueness assumptions \cite{ghadimi2018approximation}, but these conditions rarely hold in practice.
With multiple lower-level solutions, the existing methods break down.

	
	
To move beyond the singleton lower-level solution set, \citet{kwon2023fully} introduced a penalty-based framework that allows multiple follower solutions. Building on this idea, \citet{chen2024finding} obtained a refined scheme with near-optimal convergence. However, ensuring smoothness of the induced hyper-objective still demands strong regularity: The penalized model function  $h_{\sigma}(\x,\y)\coloneqq\sigma F(\x,\y)+f(\x,\y)$ must satisfy, uniformly in  $\sigma \in [0, \bar\sigma]$, an error bound or a Polyak–\L{}ojasiewicz (P\L{}) condition in
$\y$. Such requirements are often unrealistic in practice, as $F$ and $f$ typically have mismatched structures. More fundamentally, the Kurdyka–\L{}ojasiewicz (K\L{}) exponent is not preserved under summation \citep{jiang2022holderian}, so smoothness of 
hyper-objective functions cannot be guaranteed.

Without relying on these stringent conditions, \citet{chen2023bilevel} and \citet{khanduri2025doubly} proposed algorithms for nonsmooth hyper-objectives; however, by their zero-respecting nature (cf. \citep[Thm.~3.2]{chen2024finding}), they cannot in general approximate hyper-stationary points and thus only guarantee convergence to (approximate) Goldstein stationary points \citep{goldstein1977optimization}—a relatively weak notion. 
By contrast, a separate line of work studies alternative stationarity concepts via reformulations \citep{liu2022bome,xiao2023alternating,lu2023slm,abolfazli2025perturbed,lu2024first}; yet these notions (e.g., KKT stationarity \citep[Sec.~2.1]{liu2023averaged} and penalization stationarity \citep[Sec.~4.2]{yao2024overcoming}) are posed jointly in $(\x,\y)$ and do not ensure that, for a stationary pair $(\bar{\x},\bar{\y})$, the lower-level solution $\bar\y$ actually minimizes or maximizes $F(\bar\x, \y)$ over $\cS(\bar\x)$.

Given the above discussion, existing algorithms either fail to approximate a meaningful hyper-stationary point or rely on stringent assumptions to do so.
This naturally leads to a fundamental question:
	\begin{center}
		\bf Can strong hyper-stationarity be computed in general settings \\where multiple lower-level solutions exist?
	\end{center}

Addressing this question is challenging for a simple reason chain. 
When the lower level admits multiple solutions, the induced hyper-objective is typically nonsmooth and, under standard assumptions, no better than Lipschitz continuous \citep[Corollary~6.1]{chen2023bilevel}.  At precisely this level of regularity, computing (stronger) approximate Clarke stationary points is, in general, computationally intractable \citep{kornowski2021oracle,tian2021hardness}. 
Thus Lipschitz regularity alone is too weak for algorithmic purposes, motivating new, verifiable structural conditions that make meaningful hyper-stationarity attainable.

\textbf{Our Contributions.} In this paper, we address the above challenges and show that (strong) Clarke stationarity of hyper-objective functions is computable for a broad class of BLO problems. As our key contribution, we identify a hidden weak convexity/concavity structure of the hyper-objective in nonconvex–P\L{} BLO,\footnote{That is, BLO problems with a nonconvex upper-level objective and a lower-level function satisfying the P\L{} condition.} which places the analysis within the well-studied weakly convex/concave framework. Within this setting, approximate hyper-stationarity admits a natural Clarke-subdifferential characterization that we leverage to obtain computable guarantees.

The foundation of our analysis is a new concept, \emph{set smoothness} (Definition~\ref{def:setsmooth}), which extends classical smoothness to set-valued mappings and encompasses several variational regularity notions \cite{mordukhovich2024second,dontchev2009implicit,chen2023bilevel,khan2016set}. Building on this notion, we prove two complementary statements. First, if the lower-level solution mapping is set smooth, then the optimistic (resp.\ pessimistic) hyper-objective is weakly concave (resp.\ weakly convex). Second, a broad and verifiable condition guarantees set smoothness: When the lower-level function satisfies an error bound condition—equivalently, the P\L{} condition—the solution mapping is set smooth. Together, these statements provide checkable criteria under which the hyper-objective inherits a weak convexity/concavity structure.


Once the hidden weak convexity/concavity of the hyper-objective is in place, approximate Clarke hyper-stationary points can be computed by a simple inexact zeroth-order scheme. 
In the weakly convex case, results based on the Moreau envelope \citep{davis2019stochastic,zhao2022randomized,nazari2020adaptive} provide convergence and complexity guarantees.  
For the weakly concave case, however, no existing algorithmic guarantee is known, and the absence of a Moreau-type smoothing technique makes the analysis significantly more challenging. We overcome this by developing a novel convergence proof based on a Brøndsted–Rockafellar-type approximation result \cite[Theorem 2]{robinson1999linear}, and establish, to the best of our knowledge, the first general computational guarantee for finding approximate Clarke stationary points of nonsmooth hyper-objective functions.

Overall, these developments, particularly set smoothness, provide a principled foundation for the computability of hyper-stationarity in BLO and open new avenues for other structured nonsmooth optimization problems.

\paragraph{Organization.} This paper is organized as follows. Sec.~\ref{sec:pre} collects assumptions and preliminaries. 
Sec.~\ref{sec:setsmooth} introduces set smoothness and uses it to reveal a weak convexity/concavity structure of the hyper-objective. 
Sec.~\ref{sec:algorithm} presents an inexact zeroth-order scheme and establishes convergence guarantees for computing approximate Clarke hyper-stationary points. 
Sec.~\ref{sec:conclusion} concludes with final remarks.
	
\paragraph{Notation.} The notation used in this paper is mostly standard. We use $\|\x\|$ to denote the Euclidean norm of a vector $\x$ and $\|\bA\|$ to denote the $l_2$ norm of a matrix $\bA$. We use $\B(\z,r)$ to denote the ball centering at $\z$ with radius $r$, i.e., $\{\x:\|\x-\z\|\leq r\}$. For a scalar $\alpha\in\R$ and a set $\cS\subseteq\R^n$, we use $\alpha\cdot \cS$ to denote their product $\{\alpha\x:\x\in \cS\}$. We define the distance from a vector $\x\in\R^n$ to $\cS$ by $\dist(\x,\cS)\coloneqq\min_{\z\in \cS}\|\x-\z\|$ and the projection of $\x$ onto $\cS$ by $\Pi_{\cS}(\x)\coloneqq \argmin_{\z\in \cS}\|\x-\z\|$. We use ${\rm Conv}(\cS)$ to denote the convex hull of $\cS$. For two sets $\cS_1,\cS_2\subseteq\R^n$, define their Minkowski sum by
$\cS_1+\cS_2 \coloneqq \{\x_1+\x_2:\ \x_1\in\cS_1,\ \x_2\in\cS_2\}$,
and define their Hausdorff distance (with respect to $\|\cdot\|$) by
\[
d_{\mathrm H}(\cS_1,\cS_2)
\coloneqq
\max\left\{\sup_{\x_1\in\cS_1}\dist(\x_1,\cS_2),\ \sup_{\x_2\in\cS_2}\dist(\x_2,\cS_1)\right\}.
\] For a differentiable function $g:\R^m\times\R^n\to\R$, we use $\nabla g$ to denote its gradient w.r.t. the joint variables $(\x,\y)$ and $\nabla_{\x}g$ (resp. $\nabla_{\y}g$) to denote its gradient w.r.t. $\x$ (resp. $\y$). 
	
\section{Preliminaries}\label{sec:pre}
	In this paper, we focus on nonconvex-P\L{} BLO problems and make the following assumptions:
    \vspace{1em}
	\begin{assumption}[Lower-level Functions]\label{assum:basic} \quad
		\begin{enumerate}[label={{\rm (A\arabic*).}}]
			\item The function $f$ is $L_f$-smooth and twice differentiable. Moreover, $\nabla\nabla_{\y} f$ is $H_f$-Lipschitz continuous, i.e., for all $\x_1,\x_2\in\R^m$ and $\y_1,\y_2\in\R^n$,
			\[\left\|\nabla\nabla_{\y} f(\x_1,\y_1)-\nabla\nabla_{\y} f(\x_2,\y_2)\right\|\leq H_f\left( \|\x_1-\x_2\|+\|\y_1-\y_2\|\right). \]
			\item The solution set $\cS(\x)=\argmin_{\y\in\R^n}f(\x,\y)$ is nonempty closed convex for all $\x\in\R^m$. 
			\item There exists a scalar $\tau>0$ such that  for all $\x\in\R^m$ and $\y\in\R^n$, 
			\[\dist(\y,\cS(\x))\leq\tau\|\nabla_{\y} f(\x,\y)\|.\]
		\end{enumerate}
	\end{assumption}
\vspace{0.2em}
	\begin{assumption}[Upper-level and Hyper-objective Functions]\label{assum:basic2} \quad
		\begin{enumerate}[label={{\rm (B\arabic*).}}]
			\item The function $F$ is $M_{F}$-Lipschitz continuous and $L_F$-smooth.
			\item There exists $\x^{\star}\in\R^m$ such that $\varphi_o(\x^{\star})>-\infty$ (resp. $\varphi_p(\x^{\star})<+\infty$) for the optimistic (resp. pessimistic) setting.
		\end{enumerate}
	\end{assumption}
Assumptions (A1), (A2), and (B1) are standard in BLO settings; see, e.g., \cite{hong2023two,chen2024finding,ghadimi2018approximation,xiao2023alternating,abolfazli2025perturbed} and the references therein. Assumption (B2) guarantees that the hyper-objective functions are well-defined and is imposed without loss of generality.  Assumption (A3) imposes an error bound in the lower-level variable that holds uniformly over the upper-level parameter. 
This requirement is strictly weaker than the strong convexity-in-$y$ conditions commonly used in prior work \cite{franceschi2018bilevel,ji2021bilevel,hong2023two}, as it allows the solution set $\arg\min_{\y} f(\x,\y)$ to be multi-valued. Under $L_f$–smoothnes $f$ in $\y$, (A3) implies the P\L{} inequality
 \[f(\x,\y)-\min\limits_{\y\in\R^n} f(\x,\y)\leq\frac{\tau L_f^2}{2}\|\nabla_{\y}f(\x,\y)\|^2\text{ for all }\x\in\R^m,\y\in\R^n,\]
 vice versus~\cite[Theorem 3.1]{liao2024error}. 
Hence, our setting aligns with the widely adopted nonconvex–P\L{} framework for BLO \cite{shen2023penalty,xiao2023alternating,liu2022bome}.

Under the standing assumptions, we begin with the solution mapping $\cS$, which under (A3) admits the following equivalent characterization:
	\begin{equation}\label{eq:S}
		\cS(\x)=\left\{\y \in \R^n:\nabla_\y f(\x,\y)=\bz\right\}.
	\end{equation}  
	Furthermore, under Assumption \ref{assum:basic}, the solution mapping $\cS$ as well as the hyper-objective functions $\varphi_o$ and $\varphi_p$ are the Lipschitz continuous.
	\vspace{1em}
	\begin{lemma}[Lipschitz Continuity of $\cS(\x)$]\label{le:Lipschitz}(cf. \cite[Proposition 6.1]{chen2023bilevel})
		Under Assumption \ref{assum:basic}, the lower-level solution set function is $M_{\cS}$-Lipschitz continuous with $M_{\cS}= L_f\tau$, i.e.,  for any $\x_1,\x_2\in\R^m$,
		\[d_{\mathrm H}\left(\cS(\x_1),\cS(\x_2)\right)\leq M_{\cS}\|\x_1-\x_2\|.\]
	\end{lemma}
	\vspace{0.4em}
	\begin{lemma}(cf. \cite[Proposition 5.3]{chen2023bilevel})\label{co:varphi}
		Suppose that Assumption \ref{assum:basic} and \ref{assum:basic2} hold.  For the optimistic (resp. pessimistic) setting,  $\varphi_o$ (resp. $\varphi_{p}$) is $M_{\varphi}$-Lipschitz continuous with $M_{\varphi}=M_F(1+L_f\tau)$.
	\end{lemma} 

The Lipschitz continuity of the hyper-objective functions ensures that the Clarke and Goldstein subdifferentials are well defined.
	\vspace{1em}
	\begin{definition}[Clarke Subdifferential](cf. \cite[Definition 1.1]{clarke1975generalized})
    \label{def:clarke}
		For a Lipschitz continuous function $g:\R^m\to\R$, the Clarke subdifferential of $g$ at a point $\x\in\R^m$ is defined by
		\[\partial g(\x)\coloneqq{\rm Conv}\left(\left\{\s\in \R^m:\exists~\x^{\prime}\to\x,\ \nabla g(\x^{\prime})\text{ exists, }\nabla g(\x^{\prime})\to \s\right\}\right).\]
We say that $\x$ is an $(\epsilon,\delta)$-approximate Clarke stationary point of $g$ if
    \[ \dist\left(\bz,\bigcup_{\z\in\B(\x,\delta)}\partial g(\z)\right)\leq\epsilon.\]
	\end{definition}
\begin{remark}
The Clarke subdifferential $\partial g$ reduces to the gradient $\nabla g$ when $g$ is smooth. Moreover, if $g$ is convex, then $\partial g(\x)$ coincides with the vanilla subgradients defined by $\{\s:g(\z)\geq g(\x)+\s^T(\z-\x)\ \forall\ \z\in\R^m\}$.
\end{remark}    

Then, the Goldstein $\delta$-subdifferential at $\x$ can be constructed by the convex hull of the Clarke subdifferentials taken over a $\delta$-neighborhood of $\x$. Here is the formal definition of Goldstein $\delta$-subdifferential. 
		\vspace{1em}
\begin{definition}[Goldstein $\delta$-Subdifferential](cf.  \cite[Definition 2.2]{goldstein1977optimization})
For a Lipschitz continuous function $g:\R^m\to\R$ and a scalar $\delta\geq0$, the Goldstein $\delta$-subdifferential of $g$ at a point $\x\in\R^m$ is defined by
\[ \partial_{\delta}g(\x)\coloneqq{\rm Conv}\left(\left\{\bigcup_{\z\in\B(\x,\delta)}\partial g(\z)\right\}\right).\]
We say that $\x$ is an $(\epsilon,\delta)$-approximate Goldstein stationary point of $g$ if $\dist\left(\bz,\partial_{\delta}g(\x)\right)\leq\epsilon$.
\end{definition}
Leveraging the  Lipschitz continuity of the hyper-objective, recent work has established the computability of $(\epsilon,\delta)$–Goldstein hyper-stationary points \citep{chen2023bilevel}. 
However, Goldstein stationarity is strictly weaker and does not, in general, imply Clarke stationarity. 
Indeed, there exists a convex, $2$-Lipschitz function $\tilde g:\R^2\to\R$ and a point $\x$ such that $\x$ is $(0,\delta)$–Goldstein stationary while $
\min_{\z\in \mathbb{B}(\x,\delta)} \dist(\bz,\partial \tilde g(\z)) \;\ge\; \frac{2}{\sqrt{5}}$; see \citep[Proposition~2.7]{tian2022finite}. 
To obtain stronger algorithmic guarantees for hyper-objective minimization, we therefore focus on computing Clarke stationary points (and their $(\epsilon,\delta)$–approximate variants).

Despite the well-definedness, approximate Clarke stationarity is not achievable in finite time for general Lipschitz functions \cite{kornowski2021oracle,tian2021hardness}. A sufficient condition for its computability is the weak convexity of $g$ \cite{davis2019stochastic}. To elaborate on this, we review some basic properties of weakly convex functions. Given a function $g:\R^m\to\R$ and a scalar $r>0$, we say that $g$ is $r$-\textit{weakly convex} if the function $\x\mapsto g(\x)+\frac{r}{2}\|\x\|^2$ is convex. The following equivalent characterizations are useful for our analysis.
	\vspace{1em}
	\begin{lemma}[Equivalent Characterizations of Weak Convexity]\label{le:wc}(cf. \cite[Theorem 3.1]{daniilidis2005filling} and \cite[Proposition 2.2]{atenas2023unified})
		For a Lipschitz continuous function $g:\R^m\to\R$, the following statements are equivalent:   
		\begin{enumerate}[label={{\rm (\roman*)}}]
			\item $g$ is $r$-weakly convex. 
			\item For any $\theta\in[0,1]$ and $\x_1,\x_2\in\R^m$, we have 
			\[g\left(\theta\x_1+(1-\theta)\x_2\right)\leq\theta g(\x_1)+(1-\theta)g(\x_2)+\frac{r}2\theta(1-\theta)\|\x_1-\x_2\|^2.\]
			\item For any $\x_1,\x_2\in\R^m$ with $\partial g(\x_1)\neq\emptyset$, and all subgradients $\v\in\partial g(\x_1)$, we have
			\[\v^T(\x_2-\x_1)\leq g(\x_2)-g(\x_1)+\frac{r}{2}\|\x_2-\x_1\|^2.\]
		\end{enumerate}
	\end{lemma}

For an $r$–weakly convex function $g:\R^m\to\R$ with $\gamma\in(0,\tfrac{1}{r})$, we define its Moreau envelope and the proximal mapping by
\[
 g_{\gamma}(\x)\coloneqq \inf_{\z\in\R^n}\left\{g(\z)+\frac{1}{2\gamma}\|\x-\z\|^2\right\},
\qquad
\prox_{\gamma,g}(\x)\coloneqq \underset{\z\in\R^n}{\argmin}\left\{g(\z)+\frac{1}{2\gamma}\|\x-\z\|^2\right\}.
\]
Clearly, $\prox_{\gamma,g}(\x)$ is single-valued and well-defined, when $g$ is $r$-weakly convex and $\gamma<\frac1{r}$. Next, we provide the standard result, which establishes a stationarity measure based on the gradient of the Moreau envelope.
	\vspace{1em}
	\begin{lemma}[Properties of Moreau Envelope]\label{le:moreau} (cf. \cite[Proposition 2.1]{zhao2022randomized}) Suppose that $g:\R^n\to\R$ is a $r$-weakly convex function and $\gamma<\frac1{r}$. The following hold:
		\begin{enumerate}[label={{\rm (\roman*)}}]
			\item $g_{\gamma}(\x)\leq g(\x)-\frac{1-\gamma r}{2\gamma}\left\|\x-\prox_{\gamma, g}(\x)\right\|^2$.
			\item  $\gamma\dist\left(\bz,\partial g(\hat{\x})\right)\leq\left\|\x-\prox_{\gamma, g}(\x)\right\|\leq \frac{\gamma}{1-\gamma r}\dist\left(0,\partial g(\x)\right)$.
			\item $\x=\prox_{\gamma, g}(\x)$ if and only if $\bz\in\partial g(\x)$. 
			\item $\nabla  g_{\gamma}(\x)=\frac{1}{\gamma}\left(\x-\prox_{\gamma, g}(\x)\right)$. 
		\end{enumerate}
	\end{lemma}
	\vspace{0.2em}
Lemma \ref{le:moreau} (ii) and (iv) show that $\|\nabla g_{\gamma}(\x)\|$ equals zero if and only if $\x=\prox_{\gamma, g}(\x)$ and $\bz\in\partial g(\x)$. Thus $\|\nabla g_{\gamma}(\x)\|$ is a valid Clarke stationarity measure. 
Moreover, by Lemma \ref{le:moreau} (iv),  if $\|\nabla g_{\gamma}(\x)\|\le \epsilon$, then
$\|\prox_{\gamma, g}(\x)-\x\|=\gamma\|\nabla g_{\gamma}(\x)\|\le \gamma\epsilon$, hence $\dist\big(\bz,\partial g(\prox_{\gamma, g}(\x))\big)\le \epsilon$. 
Equivalently,
\begin{equation}\label{eq:measure}
		\|\nabla g_{\gamma}(\x)\|\leq\epsilon \quad \Longrightarrow\quad \dist\left(\bz,\bigcup_{\z\in\B(\x,\gamma\epsilon)}\partial g(\z)\right)\leq\epsilon. 
	\end{equation}

Since \citet{davis2019stochastic} establish non-asymptotic rates for finding $\x$ with $\|\nabla g_{\gamma}(\x)\|\le \epsilon$ when $g$ is weakly convex, \eqref{eq:measure} implies that $(\epsilon,\gamma\epsilon)$–approximate Clarke stationarity is computable in this regime. This observation motivates us to establish weak–convexity–type structure for hyper-objectives; see Sec.~\ref{sec:setsmooth}.

	Before leaving this section, we record \textit{weak concavity}, a notion closely related to weak convexity.
	For a function $g:\R^n\to\R$, we say that $g$ is $r$-weakly convex if $-g$ is $r$-weakly convex. We have the following facts.
	\vspace{1em}
	\begin{fact}\label{fact:relative}
		If a function $g:\R^m\to\R$ is $r$-weakly concave, then for any $\x_1,\x_2\in\R^m$, and $\v\in\partial g(\x_1)$, we have
		$g(\x_2)\leq g(\x_1)+\v^T(\x_2-\x_1)+\frac{r}{2}\|\x_2-\x_1\|^2.$
	\end{fact}
	\vspace{0.5em}
	\begin{fact}\label{fact:smooth}	If a function $g:\R^m\to\R$ is $r$-smooth, then $g$ is $r$-weakly convex and $r$-weakly concave, and the following inequality holds:
		\begin{equation*}\label{eq:smooth}
			\left|\theta g(\x_1)+(1-\theta)g(\x_2)-g\left(\theta\x_1+(1-\theta)\x_2\right)\right|\leq\frac{r}2\theta(1-\theta)\|\x_1-\x_2\|^2.
		\end{equation*}
	\end{fact}
	
	\section{Unveiling Hidden Structural Properties}\label{sec:setsmooth}
	This section is devoted to unveiling the hidden structural properties of the hyper-objective functions, which is the key contribution of this paper. Recall that the hyper-objective functions in \eqref{eq:varphi} are defined by minimizing/maximizing the upper-level function w.r.t. $\y$ over the parameterized set $\cS(\x)$. We are motivated to investigate the property of the set-valued function $\cS$. Inspired by the smoothness of real-valued functions, we propose a novel concept of smoothness for set-valued functions, formalized in Definition \ref{def:setsmooth}. As we will show, the lower-level solution set function $\cS$ satisfies this smoothness property, which in turn ensures the weak concavity (resp. convexity) of $\varphi_o$ (resp. $\varphi_p$).
 \vspace{0.5em}
    \begin{mdframed}[linewidth=0.5pt,roundcorner=3pt]
	\begin{definition}[Set Smoothness]\label{def:setsmooth}
    \vspace{0.5em}
		For a set-valued function $\cY:\R^m\rightrightarrows\R^n$ with a convex domain $\dom(\cY)\subseteq\R^m$, we say that it is $L$-smooth if for any $\x_1,\x_2\in\dom(\cY)$, $\theta\in[0,1]$, and all $\y\in \cY(\theta\x_1+(1-\theta)\x_2)$, there exist ${\y}_1\in \cY(\x_1)$ and ${\y}_2\in \cY(\x_2)$ such that
		\begin{equation}\label{eq:point_set}
			\|\theta{\y}_1+(1-\theta){\y}_2-\y\|\leq \frac{L}2\theta(1-\theta)\|\x_1-\x_2\|^2;
		\end{equation}
		\begin{equation}\label{eq:point_set2}
			\|{\y}_1-{\y}_2\|^2\leq {L}\|\x_1-\x_2\|^2.
		\end{equation}
        \vspace{0.1em}
	\end{definition}
    \end{mdframed}
 The condition \eqref{eq:point_set} can be viewed as a natural extension of the gradient-Lipschitz smoothness condition for real-valued functions to the setting of set-valued mappings. It guarantees that a convex combination of $\y_1 \in \cY(\x_1)$ and $\y_2 \in \cY(\x_2)$ provides a close approximation to a point in $\cY(\theta\x_1+(1-\theta)\x_2)$, with an error that decays quadratically in $\|\x_1-\x_2\|$. This yields the following set inclusion:
	\begin{equation}\label{eq:set}
		\cY\left(\theta\x_1+(1-\theta)\x_2\right)\subseteq \theta \cY(\x_1)+(1-\theta)\cY(\x_2)+\frac{L}2\theta(1-\theta)\|\x_1-\x_2\|^2\cdot\B(\bz,1).
	\end{equation} 
Intuitively, \eqref{eq:point_set2} enforces a consistent branch selection between $\cY(\x_1)$ and $\cY(\x_2)$: The chosen representatives $\y_1$ and $\y_2$ must remain aligned (Lipschitz-close) as the input varies, thereby excluding cross-branch pairings that could make the interpolation in~\eqref{eq:point_set} hold trivially while the underlying geometry is severely mismatched.

 \begin{example}[Why the condition~\eqref{eq:point_set2} is needed: A trivialization for the condition \eqref{eq:point_set}]
Define the set-valued map $\cY:\R\rightrightarrows \R^2$ by
$
\cY(x)=\{(z,x):z\in\R\}$. 
Pick $x_1=a>0$, $x_2=-a$, and $\theta=\tfrac12$; then $\theta x_1+(1-\theta)x_2=0$ and $\cY(0)=\{(z,0):z\in\R\}$.
Choose $\y=\bz\in\cY(0)$, \(\y_1=(K,a)\in\cY(a)\), and \(\y_2=(-K,-a)\in\cY(-a)\). We have
\[
\tfrac12 \y_1+\tfrac12 \y_2 \;=\; \y,
\]
so the condition~\eqref{eq:point_set} holds with zero error even though
$\|\y_1-\y_2\|=2\sqrt{a^2+K^2}$ can be made arbitrarily large as $K\to\infty$.
Hence the condition \eqref{eq:point_set} alone does not preclude severely mismatched pairings on a convex domain.
In contrast, \eqref{eq:point_set2} enforces
$\|\y_1-\y_2\|^2 \le L\|x_1-x_2\|^2 = 4La^2$, which forces $K^2\leq (L-1)a^2$
and thereby rules out such cross-branch selections unless the Lipschitz modulus is correspondingly large.
\end{example}

	With the notion of set smoothness in place, we now present our first main theoretical result. It shows that set smoothness serves as the key vehicle for establishing weak convexity/concavity of parametric optimization problems with coupled constraints: Under mild Lipschitz-type assumptions, the induced value function inherits weak convexity (or weak concavity). This is formalized in Theorem \ref{pro:setsmooth} below.
		\vspace{0.4em} 
    \begin{theorem}
	    [Implication of Set Smoothness]\label{pro:setsmooth}
		Consider a real-valued function $g:\R^m\times\R^n\to\R$ and a set-valued function $\cY:\R^m\rightrightarrows\R^n$. Let $\phi(\x)\coloneqq\max_{\y\in\cY(\x)}g(\x,\y)$ and $\cD\coloneqq\{\x:\phi(\x)>-\infty\}$. Assume that $\cD$ is a nonempty closed convex set, $\cY$ is $L_{\cY}$-smooth on $\cD$, and $g$ is $M_g$-Lipschitz continuous w.r.t. $\y$, $L_g$-smooth on $\cD\times {\rm Conv}\left(\bigcup_{\x\in\cD}\cY(\x)\right)$. Then, the function $\phi$ is $\rho$-weakly convex with $\rho=M_g{L_{\cY}}+L_g(1+ L_{\cY})$. 	
	\end{theorem}
	\vspace{0.2em}

We now instantiate the framework in the bilevel setting. Our first step is to certify \emph{set smoothness} for the lower-level solution map. Under the error-bound (EB) condition in Assumption~\ref{assum:basic} (A3), the mapping $
\mathcal{S}:\; \x \mapsto \arg\min_{\y\in\mathbb{R}^n} f(\x,\y)$
is $L_{\mathcal S}$-smooth (Theorem~\ref{pro:Swc}). 
Combining this with Theorem~\ref{pro:setsmooth} shows that the pessimistic hyper-objective $\varphi_p$ inherits weak convexity (resp. the optimistic $\varphi_o$ inherits weak concavity).

	\vspace{1em}
	\begin{theorem}[EB Implies Set Smoothness]\label{pro:Swc}
		If a function $f:\R^m\times\R^n\to\R$ satisfies Assumption \ref{assum:basic}, then its associated solution set function $\cS:\x\mapsto\argmin_{\y\in\R^n}f(\x,\y)$ is $L_{\cS}$-smooth with $L_{\cS}=\max\{2H_f\tau(1+9L_f^2\tau^2),4L^2_f\tau^2\}$.
	\end{theorem}
   \paragraph{Proof idea (why residual backfilling is essential)} 
   Fix $\x_1,\x_2$ and $\theta\in(0,1)$, and set $\x^{\theta}\coloneqq \theta\x_1+(1-\theta)\x_2$.
Given any $\y\in \cS(\x^{\theta})$, our goal is to select
$\y_1\in \cS(\x_1)$ and $\y_2\in \cS(\x_2)$ so that
\eqref{eq:point_set} and \eqref{eq:point_set2} hold.
A natural choice is to project $\y$ onto the endpoint fibers, yielding
$\bar{\y}_i \coloneqq \Pi_{\cS(\x_i)}(\y)$ for $i=1,2$.
Using Lemma~\ref{le:Lipschitz}, it is easy to see that this naive selection satisfies
\eqref{eq:point_set2}.

However, even when each fiber $\cS(\x)$ is convex, the midpoint
$\bar \y^\theta \coloneqq \theta\,\bar \y_1+(1-\theta)\,\bar \y_2$
may correspond, at $\x^\theta$, to a \emph{different local selection}
of the set-valued map $\cS(\cdot)$ than the given $\y\in \cS(\x^\theta)$.
Consequently, in general multi-solution settings the naive midpoint error
can be \emph{first-order},
\[
\|\bar\y^\theta-\y\| \;=\; \Theta\big(\|\x_1-\x_2\|\big),
\]
which motivates an additional correction to \emph{synchronize the selections}. 

We therefore \emph{align the midpoint} and \emph{backfill the residual}:
First project the naive midpoint to the middle fiber,
$\hat{\y}\coloneqq \Pi_{\cS(\x^{\theta})}(\bar\y^{\theta})$,
and then use the residual $\y-\hat \y$ to refine the endpoint representatives:
\[
\boxed{\quad
{\y}_i \;\coloneqq\; \Pi_{\cS(\x_i)}\big(\,\bar\y_i + (\y-\hat{\y})\,\big),\qquad i=1,2.\quad}
\]
This construction cancels the first-order branch mismatch in the convex
combination and leaves only a quadratic remainder. Consequently,
\eqref{eq:point_set} holds while \eqref{eq:point_set2} remains valid.

    \vspace{2mm}
    \begin{remark}
        Think of $\cS(\x)$ as a family of convex ``fibers''. The direct projections
$\bar\y_1,\bar\y_2$ may live on selections that are not synchronized with the
selection containing $\y$, so their convex combination carries a first-order drift.
Projecting $\bar\y^\theta$ to $\cS(\x^\theta)$ identifies the correct selection
at the midpoint; adding the \emph{same} residual $\y-\hat\y$ to both endpoints
moves them to the same selection as $\y$, making the first-order terms cancel
in the average and exposing the desired $\cO(\|\x_1-\x_2\|^2)$ behavior.
    \end{remark}
    
	Theorem \ref{pro:Swc} is not limited to the lower-level problem of BLO but applies to general parametric optimization problems. The established set smoothness property offers new insights into the structure of the solution mapping, which goes beyond the variational conditions considered in the literature \cite{mordukhovich2024second,khan2016set,dontchev2009implicit,chen2023bilevel,yao2023relative}. 
	\vspace{1em}
	\begin{remark}[Local Conditions are Sufficient]\label{re:Pro2bounded}
    Suppose the solution mapping $\cS$ is defined on a bounded convex domain $\cD\subseteq\R^m$.
To ensure Theorem~\ref{pro:Swc}, it suffices that Assumption~\ref{assum:basic}
holds on the set $\cD\times\cY$, where
\[
\cY = \operatorname{Conv}\!\;\left(\bigcup_{\x\in\cD}\cS(\x)\right)
\;+\; \tfrac12\,M_{\cS}\,\operatorname{diam}(\cD)\,\B(\bz,1).
\] 
	\end{remark}
The following simple example shows that the set smoothness of $\cS$ does not, in general, require Assumption~\ref{assum:basic}. This suggests that alternative sufficient conditions may guarantee set smoothness; identifying such conditions is an interesting direction for future work.

\begin{example}\label{example:simple}
Consider $f:\R^2\to\R$ defined by $f(x,y)=g(\sin x + y)$, where 
$g:\R\to\R$ has a nonempty set of minimizers 
$\cV=\arg\min_{z\in\R} g(z)$. 
Then the solution set admits a closed form:
\[
\cS(x)=\arg\min_{y\in\R} f(x,y)
=\cV-\sin x
:=\{\,v-\sin x:\ v\in\cV\,\}.
\]
In particular, $\cS$ is $1$-smooth in the sense of Definition~\ref{def:setsmooth} 
(since it is a translation of the fixed set $\cV$ by the scalar $-\sin x$), 
even though $f$ need not satisfy Assumption~\ref{assum:basic}.
\end{example}

With Theorems~\ref{pro:setsmooth} and~\ref{pro:Swc} in place, we now state our
main result on the weak convexity/concavity of the hyper-objective
$\varphi_o$ (resp.\ $\varphi_p$).

	\vspace{1em}
    \begin{theorem}[Weak convexity/concavity of the hyper-objectives]\label{th:main}
Assume Assumptions~\ref{assum:basic} and~\ref{assum:basic2}. 
Let $L_{\cS}$ be the set-smoothness modulus of $\cS$ from Theorem~\ref{pro:Swc}, and define
$
\rho \coloneqq M_F\,L_{\cS} + L_F\,(1+L_{\cS}).$
Then the following hold:
\begin{enumerate}[{\rm (i)}]
\item The optimistic hyper-objective $\varphi_o$ is $\rho$-weakly concave.
\item The pessimistic hyper-objective $\varphi_p$ is $\rho$-weakly convex.
\end{enumerate}
\end{theorem}
	\vspace{0.4em}
	\begin{proof}[Proof of Theorem \ref{th:main}]
		Theorem \ref{pro:Swc} guarantees that the set-valued function $\cS$ is $L_{\cS}$-smooth. Then, the result (ii) directly follows from Theorem \ref{pro:setsmooth}. Hence, we only need to prove (i). Note that
		\[-\varphi_o(\x) = -\min_{\y\in \cS(\x)}\ F(\x,\y)=\max_{\y\in \cS(\x)}\ -F(\x,\y).\]
		We see that $-\varphi_o$ is $\rho$-weakly convex by Theorem \ref{pro:setsmooth}. It follows that $\varphi_o$ is $\rho$-weakly concave.
	\end{proof}\vspace{0.3em}

Theorem~\ref{th:main} establishes the weak concavity/convexity of the
hyper-objectives in nonconvex--P\L{} bilevel optimization (BLO). This stands in
contrast to classical results (e.g.,~\cite[Lemma~2.2]{ghadimi2018approximation}),
which impose strong convexity of the lower level to obtain smooth
hyper-objectives. Our result is significant because it places the minimization of
these generally \emph{nonsmooth} hyper-objectives within the framework of weakly
concave/convex optimization. As a consequence, computing approximate Clarke
hyper-stationary points becomes tractable—an avenue we pursue in the next
section. Crucially, all of these developments hinge on the \emph{set smoothness}
property (Definition~\ref{def:setsmooth}), highlighting the utility of this notion.
\vspace{0.4em}
    
	\begin{remark}[Lower-level Constraints Matter]	Under Assumptions~\ref{assum:basic} and~\ref{assum:basic2}, imposing an
upper-level constraint $\x\in\cX\subseteq\R^m$ with $\cX$ nonempty, closed, and
convex preserves the conclusions of Theorem~\ref{th:main}: The functions
$\varphi_o(\x)+\iota_{\cX}(\x)$ and $\varphi_p(\x)+\iota_{\cX}(\x)$ remain
weakly concave and weakly convex, respectively, where
$\iota_{\cX}$ denotes the indicator of $\cX$. In contrast, adding a lower-level constraint $\y\in\cY$ can destroy the weak
concavity/convexity of the hyper-objectives, because the set smoothness of
$\cS$ may fail in this case; see Example~\ref{example:y}. Developing structural
conditions that recover such properties for lower-level constrained BLO is an
interesting direction for future work.\end{remark}
	\vspace{0.5em}

    \begin{example}\label{example:y}
Let $\cY=[0,1]\times[0,1]$. Consider the pessimistic bilevel problem \emph{with a lower-level constraint}:
\begin{equation}\label{eq:example}	\begin{array}{rl}
\min\limits_{x\in\R}\max\limits_{\y\in\R^2}&\ -\1^{\top}\y\\
{\rm s.t.}&\ \y\in\argmin\limits_{\y^{\prime}\in\cY}\ \|\y^{\prime}-(x,2)\|^2.	\end{array}	\end{equation}	
Assumptions~\ref{assum:basic} and~\ref{assum:basic2} 
are directly satisfied for \eqref{eq:example}, except 
\emph{for the unconstrained lower level}; the only deviation here is the added constraint $\y\in\cY$.

The lower-level solution set is the projection of $(x,2)$ onto the box $\cY$, hence
\[
\cS(x)=
\begin{cases}
\{(0,1)\}, & x\le 0,\\
\{(x,1)\}, & 0\le x\le 1,\\
\{(1,1)\}, & x\ge 1.
\end{cases}
\]
Therefore the pessimistic hyper-objective is
\[
\varphi_p(x)\;
=\begin{cases}
-1, & x\le 0,\\
-x-1, & 0\le x\le 1,\\
-2, & x\ge 1.
\end{cases}
\]
This function is \emph{not} weakly convex. Indeed, for any $\rho\ge 0$ consider 
$h_\rho(x):=\varphi_p(x)+\tfrac{\rho}{2}x^2$. Then $h_\rho$ has left and right derivatives at $x=0$ given by
$h_\rho'(0^-)=0$ and $h_\rho'(0^+)=-1$ (the quadratic term has zero slope at $0$),
which violates the monotonicity of one-sided derivatives required by convexity.
Hence no $\rho$ makes $h_\rho$ convex, i.e., $\varphi_p$ is not weakly convex.
\end{example}

	\section{Computing Approximate Clarke Hyper-stationarity}\label{sec:algorithm}
Equipped with the weak convexity/concavity of the hyper-objectives, our next goal
is to establish the computability of Clarke stationary points. First-order methods are impractical here because subgradients of the hyper-objectives are
typically unavailable. In contrast, under mild conditions—e.g., $F(\x,\cdot)$ is
concave (resp.\ convex) so that the inner maximization (resp.\ minimization) is tractable, the \emph{function values} of the hyper-objectives can be (approximately)
evaluated at a given $\x$
\cite{dutta2020algorithms,shehu2021inertial}. This motivates the use of
\emph{zeroth-order} methods for minimizing hyper-objectives
\cite{chen2023bilevel,lin2022gradient}. In particular, we adopt the inexact
zeroth-order method (IZOM) in Algorithm~\ref{al:1}, which employs a  deterministic subroutine
$\cA$ to approximately evaluate $\varphi_\beta(\x)$ (with additive accuracy $w$)
by solving the inner problem in \eqref{eq:varphi}; see
\cite{chen2023bilevel,jiang2022generalized,jiang2024near,kaushik2021method} for
practical implementations of $\cA$.

	\begin{algorithm}[htbp] 
		\caption{Inexact Zeroth-order Method (cf. \cite[Algorithm 2]{chen2023bilevel}) }
		\begin{algorithmic}
			\Require{ Radius $\varepsilon>0$, iteration number $T\in\N$, stepsize $\eta$, initial point $\x_0\in\R^m$, inexact error $w>0$, and mode parameter $\beta\in\{1,0\}$}
			\For{$t=0,1,\ldots,T-1$}
			\State Sample $\u_t$  from the the uniform distribution on the unit sphere in $\R^m$
			\State Compute $\cA^{\beta}_{w}(\x_t+\varepsilon\u_t)$ and $\cA^{\beta}_{w}(\x_t-\varepsilon\u_t)$ by subroutine $\cA$
			\State Set $\tilde{G}(\x_t)=\frac{m}{2\varepsilon}(\cA^{\beta}_{w}(\x_t+\varepsilon\u_t)-\cA^{\beta}_{w}(\x_t-\varepsilon\u_t) )\u_t$ 
			\State $\x_{t+1}=\x_t-\eta \tilde{G}(\x_t)$
			\EndFor{}
			\Ensure{$\bar{\x}$ uniformly chosen from $\left\{\x_t\right\}_{t=0}^{T-1}$}
		\end{algorithmic}\label{al:1}
	\end{algorithm}
	\begin{algorithm}[htbp]
		\caption{Deterministic Subroutine $\cA$}
		\begin{algorithmic}
			\Require{ Accuracy $w>0$, iterate point $\x\in\R^m$, and mode $\beta\in\{1,0\}$}
			\If{$\beta=1$}
			\State Compute a value $\tilde{\varphi}(\x)$ satisfying $|\tilde{\varphi}(\x)-\varphi_o(\x)|\leq w$
			\Else 
			\State Compute a value $\tilde{\varphi}(\x)$ satisfying $|\tilde{\varphi}(\x)-\varphi_p(\x)|\leq w$
			\EndIf
			\Ensure{$\cA^{\beta}_{w}(\x)=\tilde{\varphi}(\x)$}
		\end{algorithmic} \label{al:cA}
	\end{algorithm}

 Let $\varphi_{p,\gamma}$ denote the Moreau envelope of $\varphi_p$ with parameter
$\gamma$. We quantify hyper-stationarity as follows: In the optimistic case we
use the approximate Clarke stationarity measure, i.e., Definition \ref{def:clarke}, 
while in the pessimistic case we use the gradient norm of the envelope, i.e.,
$\|\nabla \varphi_{p,\gamma}(\x)\|$. These two criteria can be unified in principle via \eqref{eq:measure}; in
either form they are strictly stronger than the Goldstein stationarity measure; see Sec.~\ref{sec:pre}.
 
	We then present the main theorem of this section.
	\vspace{1em}
	\begin{theorem}\label{th:compute}
		Suppose that Assumptions \ref{assum:basic} and \ref{assum:basic2} hold. Given an iteration number $T\in\N$, set $\eta=\Theta(m^{-\frac12}T^{-\frac12}), \varepsilon=\cO(T^{-\frac12}), w=\cO({m^{-\frac34}T^{-\frac34}})$ for Algorithm \ref{al:1}. 
		Then, the following hold:
		\begin{enumerate}[{\rm (i)}]
			\item Let $\Delta_o\coloneqq\varphi_o(\x_0)-\min_{\x} \varphi_o(\x)+2M_{\varphi}\varepsilon$ with $M_{\varphi}$ given in Lemma \ref{co:varphi}. For optimistic BLO, we have 
			\[\E\left[\dist\left(\bz,{\textstyle \bigcup}_{\z\in\B\left(\bar\x,\delta\right)}\partial \varphi_o(\z)\right)^2\right]=\cO\left(\frac{\sqrt{m}(\Delta_o+1)}{\sqrt{T}}\right)\text{ with }\delta=\cO\left({T^{-\frac14}}\right). \] 
			\item Let $\gamma\in(0,\tfrac{1}{\rho+1})$ with $\rho>0$ given in Theorem \ref{th:main}, and $\Delta_p\coloneqq\varphi_{p,\gamma}(\x_0)-\min_{\x} \varphi_{p,\gamma}(\x)$. For pessimistic BLO, we have
			\[\E[\|\nabla \varphi_{p,\gamma}(\bar\x)\|^2]=\cO\left({\frac{\sqrt{m}(\Delta_p+1)}{\sqrt{T}}}\right).\]
		\end{enumerate}
	\end{theorem}
	\vspace{0.4em}
	Theorem \ref{th:compute} demonstrates, for the first time, that approximate Clarke hyper-stationarity is computable for nonconvex-P\L{} BLO in both optimistic and pessimistic settings. This result significantly improves the existing computational guarantees for nonsmooth hyper-objective functions, which are mainly based on the Goldstein stationarity \cite{chen2023bilevel,khanduri2025doubly}. The proof of the optimistic case relies on a Br{\o}ndsted-Rockafellar-like relation, details of which can be found in Appendix \ref{subsec:optimistic}. 
	
	\section{Conclusion and Discussion}\label{sec:conclusion}
	In this paper, we established the first theoretical guarantee for computing approximate Clarke hyper-stationarity in nonconvex-P\L{} BLO. The key step is unveiling the hidden structural properties of hyper-objective functions via the newly introduced smoothness concept for set-valued functions. Specifically, we proved that (i) the smoothness of the set-valued function $\cY$ ensures the weak convexity of the function $\x\mapsto\max_{\y\in\cY(\x)}\phi(\x,\y)$; and  (ii) the lower-level solution set function of BLO satisfies set smoothness. Consequently, we obtained the weak convexity/concavity of hyper-objective functions. With these properties in hand, we showed that an inexact zeroth-order method can compute approximate Clarke stationary points of hyper-objective functions.
	
	We believe that our developments contribute to a deeper understanding of the computability properties of BLO and open up several directions for future research. First, with the established structural properties, our work calls for designing faster algorithms for computing Clarke hyper-stationarity. Second, it would be valuable to generalize our methodology to establish adapted properties for BLO in other settings (e.g., structured lower-level constrained BLO \cite{khanduri2025doubly}). Furthermore, our set smoothness property, along with Theorem \ref{pro:setsmooth}, may find applications in other fields such as coupled minmax optimization \cite{Tsaknakis2023Minimax} and set-valued optimization \cite{khan2016set}, where set-valued functions play a central role.

    \section*{Ackonwledgements}
   Jiajin Li was supported by a Natural Sciences and Engineering Research Council of Canada Discovery Grant RGPIN-2025-05817. Anthony Man-Cho So was supported in part by the Hong Kong Research Grants Council (RGC) General Research Fund (GRF) project CUHK 14204823.
	\bibliographystyle{abbrvnat}    
	\bibliography{ref}
	
	\newpage
	\section*{NeurIPS Paper Checklist}
	\begin{enumerate}
		
		\item {\bf Claims}
		\item[] Question: Do the main claims made in the abstract and introduction accurately reflect the paper's contributions and scope?
		\item[] Answer: \answerYes{} 
		\item[] Justification: The developments established in Sec. \ref{sec:setsmooth}, \ref{sec:algorithm} support the claims made in the abstract and introduction.
		\item[] Guidelines:
		\begin{itemize}
			\item The answer NA means that the abstract and introduction do not include the claims made in the paper.
			\item The abstract and/or introduction should clearly state the claims made, including the contributions made in the paper and important assumptions and limitations. A No or NA answer to this question will not be perceived well by the reviewers. 
			\item The claims made should match theoretical and experimental results, and reflect how much the results can be expected to generalize to other settings. 
			\item It is fine to include aspirational goals as motivation as long as it is clear that these goals are not attained by the paper. 
		\end{itemize}
		
		\item {\bf Limitations}
		\item[] Question: Does the paper discuss the limitations of the work performed by the authors?
		\item[] Answer: \answerYes{} 
		\item[] Justification: We used Example \ref{example:y} to show that our work does not apply to the lower-level constrained case. In Sec. \ref{sec:conclusion}, we gave some remarks that reflect the limitations of our work and call for future research.
		\item[] Guidelines:
		\begin{itemize}
			\item The answer NA means that the paper has no limitation while the answer No means that the paper has limitations, but those are not discussed in the paper. 
			\item The authors are encouraged to create a separate "Limitations" section in their paper.
			\item The paper should point out any strong assumptions and how robust the results are to violations of these assumptions (e.g., independence assumptions, noiseless settings, model well-specification, asymptotic approximations only holding locally). The authors should reflect on how these assumptions might be violated in practice and what the implications would be.
			\item The authors should reflect on the scope of the claims made, e.g., if the approach was only tested on a few datasets or with a few runs. In general, empirical results often depend on implicit assumptions, which should be articulated.
			\item The authors should reflect on the factors that influence the performance of the approach. For example, a facial recognition algorithm may perform poorly when image resolution is low or images are taken in low lighting. Or a speech-to-text system might not be used reliably to provide closed captions for online lectures because it fails to handle technical jargon.
			\item The authors should discuss the computational efficiency of the proposed algorithms and how they scale with dataset size.
			\item If applicable, the authors should discuss possible limitations of their approach to address problems of privacy and fairness.
			\item While the authors might fear that complete honesty about limitations might be used by reviewers as grounds for rejection, a worse outcome might be that reviewers discover limitations that aren't acknowledged in the paper. The authors should use their best judgment and recognize that individual actions in favor of transparency play an important role in developing norms that preserve the integrity of the community. Reviewers will be specifically instructed to not penalize honesty concerning limitations.
		\end{itemize}
		
		\item {\bf Theory assumptions and proofs}
		\item[] Question: For each theoretical result, does the paper provide the full set of assumptions and a complete (and correct) proof?
		\item[] Answer:  \answerYes{} 
		\item[] Justification: We provided the full set of assumptions for our theoretical results, as detailed in Sec. \ref{sec:setsmooth} and \ref{sec:algorithm}. Complete proofs can be found in the appendix.
		\item[] Guidelines:
		\begin{itemize}
			\item The answer NA means that the paper does not include theoretical results. 
			\item All the theorems, formulas, and proofs in the paper should be numbered and cross-referenced.
			\item All assumptions should be clearly stated or referenced in the statement of any theorems.
			\item The proofs can either appear in the main paper or the supplemental material, but if they appear in the supplemental material, the authors are encouraged to provide a short proof sketch to provide intuition. 
			\item Inversely, any informal proof provided in the core of the paper should be complemented by formal proofs provided in appendix or supplemental material.
			\item Theorems and Lemmas that the proof relies upon should be properly referenced. 
		\end{itemize}
		
		\item {\bf Experimental result reproducibility}
		\item[] Question: Does the paper fully disclose all the information needed to reproduce the main experimental results of the paper to the extent that it affects the main claims and/or conclusions of the paper (regardless of whether the code and data are provided or not)?
		\item[] Answer: \answerNA{} 
		\item[] Justification: This paper focuses on the theoretical computability of hyper-stationarity for BLO and does not include experiments.
		\item[] Guidelines:
		\begin{itemize}
			\item The answer NA means that the paper does not include experiments.
			\item If the paper includes experiments, a No answer to this question will not be perceived well by the reviewers: Making the paper reproducible is important, regardless of whether the code and data are provided or not.
			\item If the contribution is a dataset and/or model, the authors should describe the steps taken to make their results reproducible or verifiable. 
			\item Depending on the contribution, reproducibility can be accomplished in various ways. For example, if the contribution is a novel architecture, describing the architecture fully might suffice, or if the contribution is a specific model and empirical evaluation, it may be necessary to either make it possible for others to replicate the model with the same dataset, or provide access to the model. In general. releasing code and data is often one good way to accomplish this, but reproducibility can also be provided via detailed instructions for how to replicate the results, access to a hosted model (e.g., in the case of a large language model), releasing of a model checkpoint, or other means that are appropriate to the research performed.
			\item While NeurIPS does not require releasing code, the conference does require all submissions to provide some reasonable avenue for reproducibility, which may depend on the nature of the contribution. For example
			\begin{enumerate}
				\item If the contribution is primarily a new algorithm, the paper should make it clear how to reproduce that algorithm.
				\item If the contribution is primarily a new model architecture, the paper should describe the architecture clearly and fully.
				\item If the contribution is a new model (e.g., a large language model), then there should either be a way to access this model for reproducing the results or a way to reproduce the model (e.g., with an open-source dataset or instructions for how to construct the dataset).
				\item We recognize that reproducibility may be tricky in some cases, in which case authors are welcome to describe the particular way they provide for reproducibility. In the case of closed-source models, it may be that access to the model is limited in some way (e.g., to registered users), but it should be possible for other researchers to have some path to reproducing or verifying the results.
			\end{enumerate}
		\end{itemize}

		\item {\bf Open access to data and code}
		\item[] Question: Does the paper provide open access to the data and code, with sufficient instructions to faithfully reproduce the main experimental results, as described in supplemental material?
		\item[] Answer: \answerNA{} 
		\item[] Justification: This paper focuses on the theoretical computability of hyper-stationarity for BLO and does not include experiments requiring code.
		\item[] Guidelines:
		\begin{itemize}
			\item The answer NA means that paper does not include experiments requiring code.
			\item Please see the NeurIPS code and data submission guidelines (\url{https://nips.cc/public/guides/CodeSubmissionPolicy}) for more details.
			\item While we encourage the release of code and data, we understand that this might not be possible, so “No” is an acceptable answer. Papers cannot be rejected simply for not including code, unless this is central to the contribution (e.g., for a new open-source benchmark).
			\item The instructions should contain the exact command and environment needed to run to reproduce the results. See the NeurIPS code and data submission guidelines (\url{https://nips.cc/public/guides/CodeSubmissionPolicy}) for more details.
			\item The authors should provide instructions on data access and preparation, including how to access the raw data, preprocessed data, intermediate data, and generated data, etc.
			\item The authors should provide scripts to reproduce all experimental results for the new proposed method and baselines. If only a subset of experiments are reproducible, they should state which ones are omitted from the script and why.
			\item At submission time, to preserve anonymity, the authors should release anonymized versions (if applicable).
			\item Providing as much information as possible in supplemental material (appended to the paper) is recommended, but including URLs to data and code is permitted.
		\end{itemize}

		\item {\bf Experimental setting/details}
		\item[] Question: Does the paper specify all the training and test details (e.g., data splits, hyperparameters, how they were chosen, type of optimizer, etc.) necessary to understand the results?
		\item[] Answer: \answerNA{} 
		\item[] Justification: This paper focuses on the theoretical computability of hyper-stationarity for BLO and does not include experiments.
		\item[] Guidelines:
		\begin{itemize}
			\item The answer NA means that the paper does not include experiments.
			\item The experimental setting should be presented in the core of the paper to a level of detail that is necessary to appreciate the results and make sense of them.
			\item The full details can be provided either with the code, in appendix, or as supplemental material.
		\end{itemize}
		
		\item {\bf Experiment statistical significance}
		\item[] Question: Does the paper report error bars suitably and correctly defined or other appropriate information about the statistical significance of the experiments?
		\item[] Answer: \answerNA{} 
		\item[] Justification: This paper focuses on the theoretical computability of hyper-stationarity for BLO and does not include experiments.
		\item[] Guidelines:
		\begin{itemize}
			\item The answer NA means that the paper does not include experiments.
			\item The authors should answer "Yes" if the results are accompanied by error bars, confidence intervals, or statistical significance tests, at least for the experiments that support the main claims of the paper.
			\item The factors of variability that the error bars are capturing should be clearly stated (for example, train/test split, initialization, random drawing of some parameter, or overall run with given experimental conditions).
			\item The method for calculating the error bars should be explained (closed form formula, call to a library function, bootstrap, etc.)
			\item The assumptions made should be given (e.g., Normally distributed errors).
			\item It should be clear whether the error bar is the standard deviation or the standard error of the mean.
			\item It is OK to report 1-sigma error bars, but one should state it. The authors should preferably report a 2-sigma error bar than state that they have a 96\% CI, if the hypothesis of Normality of errors is not verified.
			\item For asymmetric distributions, the authors should be careful not to show in tables or figures symmetric error bars that would yield results that are out of range (e.g. negative error rates).
			\item If error bars are reported in tables or plots, The authors should explain in the text how they were calculated and reference the corresponding figures or tables in the text.
		\end{itemize}
		
		\item {\bf Experiments compute resources}
		\item[] Question: For each experiment, does the paper provide sufficient information on the computer resources (type of compute workers, memory, time of execution) needed to reproduce the experiments?
		\item[] Answer: \answerNA{} 
		\item[] Justification: This paper focuses on the theoretical computability of hyper-stationarity for BLO and does not include experiments.
		\item[] Guidelines:
		\begin{itemize}
			\item The answer NA means that the paper does not include experiments.
			\item The paper should indicate the type of compute workers CPU or GPU, internal cluster, or cloud provider, including relevant memory and storage.
			\item The paper should provide the amount of compute required for each of the individual experimental runs as well as estimate the total compute. 
			\item The paper should disclose whether the full research project required more compute than the experiments reported in the paper (e.g., preliminary or failed experiments that didn't make it into the paper). 
		\end{itemize}
		
		\item {\bf Code of ethics}
		\item[] Question: Does the research conducted in the paper conform, in every respect, with the NeurIPS Code of Ethics \url{https://neurips.cc/public/EthicsGuidelines}?
		\item[] Answer: \answerYes{} 
		\item[] Justification: The authors are sure that the research conducted in the paper conforms with the NeurIPS Code of Ethics. 
		\item[] Guidelines:
		\begin{itemize}
			\item The answer NA means that the authors have not reviewed the NeurIPS Code of Ethics.
			\item If the authors answer No, they should explain the special circumstances that require a deviation from the Code of Ethics.
			\item The authors should make sure to preserve anonymity (e.g., if there is a special consideration due to laws or regulations in their jurisdiction).
		\end{itemize}

		\item {\bf Broader impacts}
		\item[] Question: Does the paper discuss both potential positive societal impacts and negative societal impacts of the work performed?
		\item[] Answer: \answerNA{} 
		\item[] Justification: This paper considers the theoretical computability of BLO. We do not see any potential negative societal impacts.
		\item[] Guidelines:
		\begin{itemize}
			\item The answer NA means that there is no societal impact of the work performed.
			\item If the authors answer NA or No, they should explain why their work has no societal impact or why the paper does not address societal impact.
			\item Examples of negative societal impacts include potential malicious or unintended uses (e.g., disinformation, generating fake profiles, surveillance), fairness considerations (e.g., deployment of technologies that could make decisions that unfairly impact specific groups), privacy considerations, and security considerations.
			\item The conference expects that many papers will be foundational research and not tied to particular applications, let alone deployments. However, if there is a direct path to any negative applications, the authors should point it out. For example, it is legitimate to point out that an improvement in the quality of generative models could be used to generate deepfakes for disinformation. On the other hand, it is not needed to point out that a generic algorithm for optimizing neural networks could enable people to train models that generate Deepfakes faster.
			\item The authors should consider possible harms that could arise when the technology is being used as intended and functioning correctly, harms that could arise when the technology is being used as intended but gives incorrect results, and harms following from (intentional or unintentional) misuse of the technology.
			\item If there are negative societal impacts, the authors could also discuss possible mitigation strategies (e.g., gated release of models, providing defenses in addition to attacks, mechanisms for monitoring misuse, mechanisms to monitor how a system learns from feedback over time, improving the efficiency and accessibility of ML).
		\end{itemize}
		
		\item {\bf Safeguards}
		\item[] Question: Does the paper describe safeguards that have been put in place for responsible release of data or models that have a high risk for misuse (e.g., pretrained language models, image generators, or scraped datasets)?
		\item[] Answer: \answerNA{} 
		\item[] Justification: This paper focuses on the theoretical computability of hyper-stationarity for BLO and poses no such risks.
		\item[] Guidelines:
		\begin{itemize}
			\item The answer NA means that the paper poses no such risks.
			\item Released models that have a high risk for misuse or dual-use should be released with necessary safeguards to allow for controlled use of the model, for example by requiring that users adhere to usage guidelines or restrictions to access the model or implementing safety filters. 
			\item Datasets that have been scraped from the Internet could pose safety risks. The authors should describe how they avoided releasing unsafe images.
			\item We recognize that providing effective safeguards is challenging, and many papers do not require this, but we encourage authors to take this into account and make a best faith effort.
		\end{itemize}
		
		\item {\bf Licenses for existing assets}
		\item[] Question: Are the creators or original owners of assets (e.g., code, data, models), used in the paper, properly credited and are the license and terms of use explicitly mentioned and properly respected?
		\item[] Answer: \answerNA{} 
		\item[] Justification: This paper focuses on the theoretical computability of hyper-stationarity for BLO and does not use existing assets.
		\item[] Guidelines:
		\begin{itemize}
			\item The answer NA means that the paper does not use existing assets.
			\item The authors should cite the original paper that produced the code package or dataset.
			\item The authors should state which version of the asset is used and, if possible, include a URL.
			\item The name of the license (e.g., CC-BY 4.0) should be included for each asset.
			\item For scraped data from a particular source (e.g., website), the copyright and terms of service of that source should be provided.
			\item If assets are released, the license, copyright information, and terms of use in the package should be provided. For popular datasets, \url{paperswithcode.com/datasets} has curated licenses for some datasets. Their licensing guide can help determine the license of a dataset.
			\item For existing datasets that are re-packaged, both the original license and the license of the derived asset (if it has changed) should be provided.
			\item If this information is not available online, the authors are encouraged to reach out to the asset's creators.
		\end{itemize}
		
		\item {\bf New assets}
		\item[] Question: Are new assets introduced in the paper well documented and is the documentation provided alongside the assets?
		\item[] Answer: \answerNA{} 
		\item[] Justification: This paper focuses on the theoretical computability of hyper-stationarity for BLO and does not release new assets.
		\item[] Guidelines:
		\begin{itemize}
			\item The answer NA means that the paper does not release new assets.
			\item Researchers should communicate the details of the dataset/code/model as part of their submissions via structured templates. This includes details about training, license, limitations, etc. 
			\item The paper should discuss whether and how consent was obtained from people whose asset is used.
			\item At submission time, remember to anonymize your assets (if applicable). You can either create an anonymized URL or include an anonymized zip file.
		\end{itemize}
		
		\item {\bf Crowdsourcing and research with human subjects}
		\item[] Question: For crowdsourcing experiments and research with human subjects, does the paper include the full text of instructions given to participants and screenshots, if applicable, as well as details about compensation (if any)? 
		\item[] Answer: \answerNA{} 
		\item[] Justification: This paper focuses on the theoretical computability of hyper-stationarity for BLO and does not involve crowdsourcing nor research with human subjects.
		\item[] Guidelines:
		\begin{itemize}
			\item The answer NA means that the paper does not involve crowdsourcing nor research with human subjects.
			\item Including this information in the supplemental material is fine, but if the main contribution of the paper involves human subjects, then as much detail as possible should be included in the main paper. 
			\item According to the NeurIPS Code of Ethics, workers involved in data collection, curation, or other labor should be paid at least the minimum wage in the country of the data collector. 
		\end{itemize}
		
		\item {\bf Institutional review board (IRB) approvals or equivalent for research with human subjects}
		\item[] Question: Does the paper describe potential risks incurred by study participants, whether such risks were disclosed to the subjects, and whether Institutional Review Board (IRB) approvals (or an equivalent approval/review based on the requirements of your country or institution) were obtained?
		\item[] Answer: \answerNA{} 
		\item[] Justification: This paper focuses on the theoretical computability of hyper-stationarity for BLO and does not involve crowdsourcing nor research with human subjects.
		\item[] Guidelines:
		\begin{itemize}
			\item The answer NA means that the paper does not involve crowdsourcing nor research with human subjects.
			\item Depending on the country in which research is conducted, IRB approval (or equivalent) may be required for any human subjects research. If you obtained IRB approval, you should clearly state this in the paper. 
			\item We recognize that the procedures for this may vary significantly between institutions and locations, and we expect authors to adhere to the NeurIPS Code of Ethics and the guidelines for their institution. 
			\item For initial submissions, do not include any information that would break anonymity (if applicable), such as the institution conducting the review.
		\end{itemize}
		
		\item {\bf Declaration of LLM usage}
		\item[] Question: Does the paper describe the usage of LLMs if it is an important, original, or non-standard component of the core methods in this research? Note that if the LLM is used only for writing, editing, or formatting purposes and does not impact the core methodology, scientific rigorousness, or originality of the research, declaration is not required.
		\item[] Answer: \answerNA{} 
		\item[] Justification: This paper focuses on the theoretical computability of hyper-stationarity for BLO and does not involve LLMs as any important, original, or non-standard components.
		\item[] Guidelines:
		\begin{itemize}
			\item The answer NA means that the core method development in this research does not involve LLMs as any important, original, or non-standard components.
			\item Please refer to our LLM policy (\url{https://neurips.cc/Conferences/2025/LLM}) for what should or should not be described.
		\end{itemize}
		
	\end{enumerate}
    \newpage
	\appendix
	\section{Missing Proofs for Sec. \ref{sec:pre}}
	\subsection{Proof of Lemma \ref{le:Lipschitz}}
\begin{proof}
Thanks to Assumption \ref{assum:basic} (A3), we have 
    \[\dist(\y_1,\cS(\x_2))\leq \tau\|\nabla_{\y}f(\x_2,\y_1)\|,\]
for any $\y_1\in \cS(\x_1)$. 
	Moreover, we have
	\[\|\nabla_{\y}f(\x_2,\y_1)\|=\|\nabla_{\y}f(\x_2,\y_1)-\nabla_{\y}f(\x_1,\y_1)\|\leq L_f\|\x_1-\x_2\|,\]
	where the equality follows from $\nabla_{\y} f(\x_1,\y_1)=\bz$, and the inequality is due to the $L_f$-smoothness of $f$. Putting them together yields $\dist(\y_1,\cS(\x_2))\leq L_f\tau\|\x_1-\x_2\|$ for all $\y_1\in S(\x_1)$. 
    
    By the same argument with $\x_1$ and $\x_2$ interchanged, we have $\dist(\y_2,\cS(\x_1))\leq L_f\tau\|\x_1-\x_2\|$ for all $\y_2\in \cS(\x_1)$. By the definition of $d_{\mathrm H}(\cS(\x_1),\cS(\x_2))$, we conclude 
	\[\begin{aligned}
		d_{\mathrm H}\left(\cS(\x_1),\cS(\x_2)\right)&=\max\left\{\sup_{\y_1\in \cS(\x_1)}\dist(\y_1,\cS(\x_2)),\sup_{\y_2\in \cS(\x_2)}\dist(\y_2,\cS(\x_1))\right\}\\
		&\leq L_f\tau\|\x_1-\x_2\|.
	\end{aligned}\]
	This completes the proof.
    \end{proof}
	\subsection{Proof of Lemma \ref{co:varphi}}
\begin{proof}
We prove the $M_\varphi$-Lipschitz continuity for $\varphi_p$ only, as the argument for $\varphi_o$ is entirely analogous. 
    
Then, it suffices to show that for any $\x_1,\x_2\in\dom(\varphi_p)$,
	\[|\varphi_p(\x_1)-\varphi_p(\x_2)|\leq M_{\varphi}\|\x_2-\x_2\|. \]
Note that $\cS(\x)$ is closed but not necessarily compact.
We can only find a sequence $\{\y_1^k\}_{k\in\N}\subseteq \cS(\x_1)$ such that $F(\x_1,\y_1^k)\to \sup_{\y\in \cS(\x_1)} F(\x_1,\y)=\varphi_p(\x_1)$. Let $\y_2^k\coloneqq \Pi_{\cS(\x_2)}(\y_1^k)$. Then, by $\y_2^k\in S(\x_2)$, we have $\varphi_p(\x_2)=\sup_{\y\in \cS(\x_2)} F(\x_2,\y)\geq F(\x_2,\y_2^k)$. This observation, combined with the $M_F$-Lipschitz continuity of $F$, yields 
	\[F(\x_1,\y_1^k)-\varphi_p(\x_2)  \leq F(\x_1,\y_1^k)-F(\x_2,\y_2^k)\leq M_F\left(\|\x_1-\x_2\|+\|\y_1^k-\y_2^k\|\right),\]
where the last inequality follows from Assumption \ref{assum:basic2} (B1).

Moreover, we have  
\[\|\y_1^k-\y_2^k\|=\|\y_1^k-\Pi_{\cS(\x_2)}(\y_1^k)\|=\dist(\y_1^k,\cS(\x_2))\leq d_{\mathrm H}(\cS(\x_1),\cS(\x_2))\leq M_{\cS}\|\x_1-\x_2\|,\]
where the last inequality is due to Lemma \ref{le:Lipschitz}. Then, we can see that for all $k\in\N$,
	\[F(\x_1,\y_1^k)-\varphi_p(\x_2)\leq M_F(1+M_{\cS})\|\x_1-\x_2\|. \]
	Letting $k\to\infty$ and recalling $F(\x_1,\y_1^k)\leftarrow\varphi_p(\x_1)$, we obtain
	\begin{equation*}\label{eq:lvarphi}
		\varphi_p(\x_1)-\varphi_p(\x_2)\leq M_F(1+M_{\cS}) \|\x_1-\x_2\|,
	\end{equation*}
	which completes the proof.
\end{proof}
    
	\subsection{Proof of Fact \ref{fact:relative}}
\begin{proof}
By \cite[Proposition~4.4]{vial1983strong}, the weakly convex function $-g$ is locally Lipschitz (hence so is $g$), so the Clarke subdifferential $\partial g$ is well defined. Moreover, the Clarke subdifferential satisfies $\partial(-g)=-\,\partial g$ \cite[Prop.~2.3.1]{clarke1990optimization}. Applying Lemma~\ref{le:wc}\,(iii) to $-g$ then yields the claim.
\end{proof}

	\section{Proof of Theorem \ref{pro:setsmooth}}
\begin{proof}
To establish the weak convexity of $\phi$, we verify the condition in Lemma~\ref{le:wc}\,(ii). Specifically, we show that there exists a constant $\rho>0$ such that, for all $\theta\in[0,1]$ and all $\x_1,\x_2\in\cD$, 
	\begin{equation}\label{eq:weak_phi}
		\phi(\x^{\theta})\leq \theta\phi(\x_1)+(1-\theta)\phi(\x_2)+\frac{\rho}2\theta(1-\theta)\|\x_1-\x_2\|^2, 
	\end{equation}
	where we let $\x^{\theta}\coloneqq\theta\x_1+(1-\theta)\x_2$ for notation convenience.
    
By definition, $\phi(\x^\theta)=\max_{\y^\prime\in\cY(\x^\theta)} g(\x^\theta,\y^\prime)$.
Hence \eqref{eq:weak_phi} will follow if we show that, for all
$\y\in\cY(\x^\theta)$,
\begin{equation}\label{eq:g_phi}
  g(\x^\theta,\y)-\theta\,\phi(\x_1)-(1-\theta)\,\phi(\x_2)
  \;\le\; \tfrac{\rho}{2}\,\theta(1-\theta)\,\|\x_1-\x_2\|^2 .
\end{equation}
Taking the maximization over $\y\in\cY(\x^\theta)$ then yields \eqref{eq:weak_phi}.

We proceed via the $L_{\cY}$-smoothness of $\cY$, which guarantees the existence of
$\y_1\in\cY(\x_1)$ and $\y_2\in\cY(\x_2)$ such that
\begin{align}
    & \left\|\theta\y_1+(1-\theta){\y}_2-\y \right\|\leq\frac{L_{\cY}}2\theta(1-\theta)\|\x_1-\x_2\|^2;\label{eq:pro1} \\
    & \left\|{\y}_1-{\y}_2\right\|^2\leq L_{\cY}\left\|\x_1-\x_2\right\|^2. \label{eq:pro} 
\end{align}
Then, 
using the  fact that $\phi(\x_i)=\max_{\y^{\prime}\in \cY(\x_i)} g(\x_i,\y^{\prime})\geq g(\x_i,\y_i)$ for $i=1,2$, we have
\begin{equation*}\label{eq:Lg}
	\begin{aligned}
			&g\left(\x^{\theta},\y\right)-\theta\phi(\x_1)-(1-\theta)\phi(\x_2)\\
			\leq& g\left(\x^{\theta},\y\right)-\theta g(\x_1,\y_1)-(1-\theta)g(\x_2,\y_2)\\
			\leq& g\left(\x^{\theta},\y\right)-g\left(\x^{\theta},\theta\y_1+(1-\theta){\y}_2\right)+\frac{L_g}{2}\theta(1-\theta)\left(\left\|\x_1-\x_2\right\|^2+\left\|\y_1-\y_2\right\|^2\right),
	\end{aligned}
\end{equation*}	
where the last inequality is due to the $L_g$-smoothness of $g$ on $\cD\times{\rm Conv}(\bigcup_{\x\in\cD}\cY(\x))$ and Fact \ref{fact:smooth}.

This, together with the $M_g$-Lipschitz continuity of $g$ w.r.t. $\y$ and the set-smoothness inequalities \eqref{eq:pro1} and \eqref{eq:pro}, yields
\begin{equation*}\label{eq:Mg}
	\begin{aligned}
			&g\left(\x^{\theta},\y \right)-\theta\phi(\x_1)-(1-\theta)\phi(\x_2)\\
		\leq& M_g\left\|\theta\y_1+(1-\theta){\y}_2-\y\right\|+\frac{L_g}{2}\theta(1-\theta)\left(\left\|\x_1-\x_2\right\|^2+\left\|\y_1-\y_2\right\|^2\right)\\
		\leq& \frac{M_gL_{\cY}}2\theta(1-\theta)\|\x_1-\x_2\|^2+\frac{L_g}{2}\theta(1-\theta)\left(\left\|\x_1-\x_2\right\|^2+L_{\cY}\|\x_1-\x_2\|^2 \right)\\
		=&  \frac{M_gL_{\cY}+L_g(1+L_{\cY})}{2}\theta(1-\theta)\|\x_1-\x_2\|^2,
	\end{aligned}
\end{equation*}
for all $\y \in \cY(\x^\theta)$. 
	We prove the desired inequality \eqref{eq:g_phi} with $\rho=M_g{L_{\cY}}+L_g(1+L_{\cY})$. 
	
    \end{proof}
	\section{Proof of Theorem \ref{pro:Swc}}
We start by stating a lemma that will be used in the sequel.
\vspace{0.2em}
	\begin{lemma}\label{le:hession}
		Under Assumption \ref{assum:basic} (A1), for any $\x_1,\x_2\in\R^m$ and $\y_1,\y_2\in\R^n$, we have
		\[\begin{aligned}
			{\rm (i)} \qquad&\left\|\nabla_{\y} f(\x_1,\y_1)-\nabla_{\y} f\left(\x_2,\y_2\right)-\nabla\nabla_{\y}f\left(\x_2,\y_2\right)\left(\x_1-\x_2,\y_1-\y_2\right)\right\|\qquad\qquad\\
			\leq& \frac{H_f}{2}\left(\left\|\x_1-\x_2\right\|^2+\left\|\y_1-\y_2\right\|^2\right);\\
			{\rm (ii)} 	\qquad&\left\|\theta\nabla_{\y} f(\x_1,\y_1)+(1-\theta)\nabla_{\y} f(\x_2,\y_2)-\nabla_{\y} f(\x^{\theta},\y^{\theta})\right\| \qquad\qquad\\
			\leq& \frac{H_f}{2}\theta(1-\theta)\left(\left\|\x_1-\x_2\right\|^2+\left\|\y_1-\y_2\right\|^2\right),\quad\forall~\theta\in[0,1],
		\end{aligned} \]
	where $\x^{\theta}\coloneqq \theta\x_1+(1-\theta)\x_2$ and $\y^{\theta}\coloneqq \theta\y_1+(1-\theta)\y_2$.
	\end{lemma}
	\begin{proof}[Proof of Lemma \ref{le:hession}]
		(i) The argument directly follows from \cite[Lemma 1]{nesterov2006cubic}:
		\begin{equation*}
			\begin{aligned}
				&\left\|\nabla_{\y} f(\x_1,\y_1)-\nabla_{\y} f\left(\x_2,\y_2\right)-\nabla\nabla_{\y}f\left(\x_2,\y_2\right)\left(\x_1-\x_2,\y_1-\y_2\right)\right\|\\
				=&\|\int_0^1\nabla\nabla_{\y}f\left(\x_1+t\left(\x_2-\x_1\right),\y_1+t\left(\y_2-\y_1\right)\right)\left(\x_1-\x_2,\y_1-\y_2\right)\textrm{d}t \\
				&\left.-\nabla\nabla_{\y}f\left(\x_2,\y_2\right)\left(\x_1-\x_2,\y_1-\y_2\right)\right\|\\
				=&\left\|\int_0^1\left(\nabla\nabla_{\y}f\left(\x_1+t\left(\x_2-\x_1\right),\y_1+t\left(\y_2-\y_1\right)\right)-\nabla\nabla_{\y}f\left(\x_2,\y_2\right)\right)\left(\x_1-\x_2,\y_1-\y_2\right)\textrm{d}t \right\|\\
				\leq&\int_0^1\left\|\left(\nabla\nabla_{\y}f\left(\x_1+t\left(\x_2-\x_1\right),\y_1+t\left(\y_2-\y_1\right)\right)-\nabla\nabla_{\y}f\left(\x_2,\y_2\right)\right)\left(\x_1-\x_2,\y_1-\y_2\right) \right\| \textrm{d}t\\
				\leq&\int_0^1\left\|\nabla\nabla_{\y}f\left(\x_1+t\left(\x_2-\x_1\right),\y_1+t\left(\y_2-\y_1\right)\right)-\nabla\nabla_{\y}f\left(\x_2,\y_2\right)\right\|\left\|\left(\x_1-\x_2,\y_1-\y_2 \right)\right\| \textrm{d}t\\
				\leq&\int_0^1H_ft\left\|\left(\x_1-\x_2,\y_1-\y_2\right)\right\|\cdot\left\|\left(\x_1-\x_2,\y_1-\y_2\right)\right\|\textrm{d}t \\
				=&\frac{H_f}2\left(\left\|\x_1-\x_2\right\|^2+\left\|\y_1-\y_2\right\|^2\right),
			\end{aligned}
		\end{equation*}
		where the first inequality is due to Jensen's inequality; the second inequality is due to the Cauchy inequality; and the third inequality uses the $H_f$-Lipschitz continuity of $\nabla\nabla_{\y}f$.
		
		(ii) Using the result of (i), we have
		\begin{equation*}
			\begin{aligned}
			&\left\|\nabla_{\y} f(\x_1,{\y}_1)-\nabla_{\y} f\left(\x^{\theta},{\y}^{\theta}\right)-\nabla\nabla_{\y}f\left(\x^{\theta},\y^{\theta}\right)\left(\x_1-\x^{\theta},\y_1-\y^{\theta}\right)\right\|\\
			\leq&\frac{H_f}2\left(\left\|\x_1-\x^{\theta}\right\|^2+\left\|\y_1-\y^{\theta}\right\|^2\right).
		\end{aligned}
		\end{equation*}
	It follows from  $\x^{\theta}= \theta\x_1+(1-\theta)\x_2$ and $\y^{\theta}= \theta\y_1+(1-\theta)\y_2$ that
	\begin{equation}\label{eq:lip_hessian1}
		\begin{aligned}
		&\left\|\nabla_{\y} f(\x_1,{\y}_1)-\nabla_{\y} f\left(\x^{\theta},{\y}^{\theta}\right)-(1-\theta)\nabla\nabla_{\y}f\left(\x^{\theta},\y^{\theta}\right)\left(\x_1-\x_2,\y_1-\y_2\right)\right\|\\
		\leq&\frac{H_f}2(1-\theta)^2\left(\left\|\x_1-\x_2\right\|^2+\left\|\y_1-\y_2\right\|^2\right).
	\end{aligned}
	\end{equation}
	Using the same arguments with $(\x_1,\y_1)$ replaced by $(\x_2,\y_2)$, we have
	\begin{equation}\label{eq:lip_hessian2}
		\begin{aligned}
			&\left\|\nabla_{\y} f(\x_2,{\y}_2)-\nabla_{\y} f\left(\x^{\theta},{\y}^{\theta}\right)-\theta\nabla\nabla_{\y}f\left(\x^{\theta},\y^{\theta}\right)\left(\x_2-\x_1,\y_2-\y_1\right)\right\|\\
			\leq&\frac{H_f}2\theta^2\left(\left\|\x_1-\x_2\right\|^2+\left\|\y_1-\y_2\right\|^2\right).
		\end{aligned}
	\end{equation}
Then, the desired inequality follows from the weighted sum $\theta\times$\eqref{eq:lip_hessian1}$+(1-\theta)\times$\eqref{eq:lip_hessian2}, and the triangle inequality that
	\begin{align*}
	&\theta\left\|\nabla_{\y} f(\x_1,{\y}_1)-\nabla_{\y} f\left(\x^{\theta},{\y}^{\theta}\right)-(1-\theta)\nabla\nabla_{\y}f\left(\x^{\theta},\y^{\theta}\right)\left(\x_1-\x_2,\y_1-\y_2\right)\right\|\\
	&+(1-\theta)\left\|\nabla_{\y} f(\x_2,{\y}_2)-\nabla_{\y} f\left(\x^{\theta},{\y}^{\theta}\right)-\theta\nabla\nabla_{\y}f\left(\x^{\theta},\y^{\theta}\right)\left(\x_2-\x_1,\y_2-\y_1\right)\right\|\\
	\geq& \left\|\theta\nabla_{\y} f(\x_1,\y_1)+(1-\theta)\nabla_{\y} f(\x_2,\y_2)-\nabla_{\y} f(\x^{\theta},\y^{\theta})\right\|.
\end{align*}
	\end{proof}
We are now ready to prove Proposition~\ref{pro:Swc}.
Fix $\theta\in[0,1]$ and $\x_1,\x_2\in\R^m$, and let  $\y\in\cS(\theta\x_1+(1-\theta)\x_2)$.
Our goal is to construct $\y_1\in\cS(\x_1)$ and $\y_2\in\cS(\x_2)$ such that
\eqref{eq:point_set} and \eqref{eq:point_set2} hold.
 
	
	\subsection{Step 1:  Choose projection points as the candidate approximation points.}
	Let $\bar{\y}_1\coloneqq \Pi_{\cS(\x_1)}(\y)$ and $\bar{\y}_2\coloneqq \Pi_{\cS(\x_2)}(\y)$. For simplicity, we write $\x^{\theta}=\theta\x_1+(1-\theta)\x_2$ and $\bar{\y}^{\theta}=\theta\bar{\y}_1+(1-\theta)\bar{\y}_2$. We start with giving basic estimates on  $\|\bar{\y}_1-\bar{\y}_2\|$ and $	\dist(\bar{\y}^{\theta},\cS(\x^{\theta}))$, in Claim \ref{claim:y1y2}. 
	\begin{claim}\label{claim:y1y2}
		Let  $L_0\coloneqq H_f\tau(1+M_{\cS}^2)/2$. The following hold:
		\begin{equation}\label{eq:y1y2}
			\|\bar{\y}_1-\bar{\y}_2\|\leq M_{\cS}\|\x_1-\x_2\|
		\end{equation}
		\begin{equation}\label{eq:L0}
\dist\left(\bar{\y}^{\theta},\cS\left(\x^{\theta}\right)\right)\leq{{L_0}}\theta(1-\theta)\|\x_1-\x_2\|^2.
		\end{equation} 
	\end{claim}
    \begin{proof}[Proof of Claim \ref{claim:y1y2}]
We have 
	\begin{equation*}
		\begin{aligned}
			\|\bar{\y}_1-\bar{\y}_2\|&\leq \|\bar{\y}_1-\y\|+\|\y-\bar{\y}_2\|\\
			& = \dist(\y,\cS(\x_1))+\dist(\y,\cS(\x_2))\\
			& \leq d_{\mathrm H}\left(\cS\left(\x^{\theta}\right),\cS(\x_1)\right)+d_{\mathrm H}\left(\cS\left(\x^{\theta}\right),\cS(\x_2)\right)\\
			&\leq M_{\cS}\|\x_1-(\theta\x_1+(1-\theta)\x_2)\|+M_{\cS}\|\x_2-(\theta\x_1+(1-\theta)\x_2)\|\\
			&=M_{\cS}\|\x_1-\x_2\|,
		\end{aligned}
	\end{equation*}
	where the second inequality follows from $\y\in \cS(\x^{\theta})$ and the definition of the Hausdorff distance; the third inequality is due to Lemma \ref{le:Lipschitz} and $\x^{\theta}=\theta\x_1+(1-\theta)\x_2$. This proves \eqref{eq:y1y2}.
	
	We then prove \eqref{eq:L0}. First, by the $H_f$-Lipschitz continuity of $\nabla\nabla_{\y} f$ and Lemma \ref{le:hession} (ii), we have
	\begin{equation}
		\begin{aligned}\label{eq:lip_weight}
			&\left\|\theta\nabla_{\y} f(\x_1,\bar{\y}_1)+(1-\theta)\nabla_{\y} f(\x_2,\bar{\y}_2)-\nabla_{\y} f(\x^{\theta},\bar{\y}^{\theta}) \right\|\\
			\leq&\frac{H_f}{2}\theta(1-\theta)\left(\left\|\x_1-\x_2\right\|^2+\left\|\bar{\y}_1-\bar{\y}_2\right\|^2\right)\\
			\leq &\frac{H_f}2\theta(1-\theta)\left(1+M_{\cS}^2\right)\|\x_1-\x_2\|^2,
		\end{aligned}
	\end{equation}
	where the  second inequality is due to \eqref{eq:y1y2}.  
	
Since $\bar \y_1\in \cS(\x_1)$ and $\bar \y_2\in \cS(\x_2)$, we have the first-order optimality
\[
\nabla_\y f(\x_1,\bar \y_1)=0, \qquad \nabla_\y f(\x_2,\bar \y_2)=0.
\]
Substituting these identities into \eqref{eq:lip_weight} yields
	\[\left\|\nabla_{\y}f\left(\x^{\theta},\bar{\y}^{\theta}\right)\right\|\leq\frac{H_f}2\theta(1-\theta)\left(1+M_{\cS}^2\right)\|\x_1-\x_2\|^2.\]
Combining the error bound from Assumption~\ref{assum:basic},
\(\dist(\bar\y^{\theta},\cS(\x^{\theta}))\le \tau\|\nabla_{\y}f(\x^{\theta},\bar\y^{\theta})\|\),
with the preceding inequality, we obtain
\[\dist\left(\bar{\y}^{\theta},\cS\left(\x^{\theta}\right)\right)\leq \frac{H_f\tau}2\theta(1-\theta)\left(1+M_{\cS}^2\right)\|\x_1-\x_2\|^2,\]
which accords with \eqref{eq:L0} with $L_0=H_f\tau(1+M_{\cS}^2)/2$. Claim \ref{claim:y1y2} is proved.  
\end{proof}
Note that \eqref{eq:L0} controls only the distance from $\bar\y^{\theta}$ to the set
$\cS(\x^{\theta})$; it does not guarantee that the specific point
$\theta\bar\y_1+(1-\theta)\bar\y_2$ lies on the same branch as the metric projection
$\Pi_{\cS(\x^{\theta})}(\bar\y^{\theta})$. Hence one cannot conclude that
$\|\theta\bar\y_1+(1-\theta)\bar\y_2-\y\|=\cO(\|\x_1-\x_2\|^{2})$ in general.
Therefore we cannot simply take $\y_1=\bar\y_1$ and $\y_2=\bar\y_2$, which motivates
the rectification in Step~2.

\subsection{Step 2: Translate the candidate approximation points} 

Let $\hat{\y}\coloneqq \Pi_{\cS(\x^{\theta})}(\bar{\y}^{\theta})$, $\hat{\y}_1\coloneqq \bar{\y}_1+(\y-\hat{\y})$, and $\hat{\y}_2\coloneqq \bar{\y}_2+(\y-\hat{\y})$. We will bound the three quantities $\|\hat\y-\y\|$, $\dist(\hat\y_1,\cS(\x_1))$, and $\dist(\hat\y_2,\cS(\x_2))$.
	\begin{claim}\label{claim:barytheta}
		Let $L_1\coloneqq H_f\tau\left(1+17M_{\cS}^2\right)/2$. The following hold:
		\begin{equation}\label{eq:barytheta}
			\left\|\hat{\y}-\y \right\|\leq 2\theta(1-\theta)M_{\cS}\left\|\x_1-\x_2\right\|;
		\end{equation}
		\begin{equation}\label{eq:distbary}
			\dist\left(\hat{\y}_1,\cS(\x_1)\right)\leq L_1 (1-\theta)^2\|\x_1-\x_2\|^2;\qquad \dist\left(\hat{\y}_2,\cS(\x_2)\right)\leq L_1\theta^2\|\x_1-\x_2\|^2.
		\end{equation}
	\end{claim}
    \begin{proof}[Proof of Claim \ref{claim:barytheta}]
 We begin with the proof of \eqref{eq:barytheta}. The non-expansiveness of the projection operator $\Pi_{\cS(\x^{\theta})}(\cdot)$ yields
	\begin{equation}\label{eq:ythetay}
		\left\|\hat{\y}-\y \right\|=\left\|\Pi_{\cS(\x^{\theta})}(\bar{\y}^{\theta})-\Pi_{\cS(\x^{\theta})}(\y)\right\|	\leq \left\|\bar{\y}^{\theta}-\y \right\|\leq\theta\left\| \bar\y_1-\y\right\|+(1-\theta)\left\|\bar\y_2-\y \right\|.
	\end{equation} 

Moreover, by $\bar{\y}_i=\Pi_{\cS(\x_i)}(\y)$ for $i=1,2$, $\y\in\cS(\x^{\theta})$, and the $M_{\cS}$-Lipschitz continuity of $\cS$ from Lemma \ref{le:Lipschitz}, we have
\begin{equation}\label{eq:y1_y}
	\|\bar{\y}_i-\y\|=\dist(\y,\cS(\x_i))
	\leq d_{\mathrm H}\left(\cS\left(\x^{\theta}\right),\cS(\x_i)\right)
	\leq M_{\cS}\left\|\x_i-\x^{
		\theta}\right\|,\quad i=1,2.\\
\end{equation}
Combining \eqref{eq:ythetay} and \eqref{eq:y1_y} leads to 
\[	\left\|\hat{\y}-\y \right\|\leq\theta M_{\cS}\left\|\x_1-\x^{
	\theta}\right\| +(1-\theta)M_{\cS}\left\|\x_2-\x^{
	\theta}\right\|=2\theta(1-\theta)M_{\cS}\left\|\x_1-\x_2\right\|.   \]
	
We then continue to prove \eqref{eq:distbary}. 
It suffices to upper bound $\|\nabla_{\y}f(\x_1,\hat\y_1)\|$, due to Assumption~\ref{assum:basic} (A3). 
By Lemma \ref{le:hession}, we have $\nabla_{\y}f(\x_1,\hat{\y}_1)$:
	\[\left\|\nabla_{\y} f(\x_1,\hat{\y}_1)-\nabla_{\y}f(\x_1,\bar{\y}_1)-\nabla\nabla_{\y}f(\x_1,\bar{\y}_1)(\bz,\hat{\y}_1-\bar{\y}_1) \right\|\leq\frac{H_f}{2}\|\hat{\y}_1-\bar{\y}_1\|^2.\]
	Noticing $\bar{\y}_1\in \cS(\x_1)$, we have $\nabla_{\y}f(\x_1,\bar{\y}_1)=\bz$, and thus
	\[\left\|\nabla_{\y} f(\x_1,\hat{\y}_1)-\nabla\nabla_{\y}f(\x_1,\bar{\y}_1)(\bz,\hat{\y}_1-\bar{\y}_1) \right\|\leq\frac{H_f}{2}\|\hat{\y}_1-\bar{\y}_1\|^2.\]
	Using the triangle inequality and the identity $\hat\y_1-\bar\y_1=\y-\hat\y$ (by the definition of $\hat\y_1$), we obtain
	\begin{equation}\label{eq:nablayf}
		\left\|\nabla_{\y} f(\x_1,\hat{\y}_1)\right\|
		\leq\|\nabla\nabla_{\y}f(\x_1,\bar{\y}_1)(\bz,\y-\hat{\y})\|+\frac{H_f}{2}\|\y-\hat{\y}\|^2.
	\end{equation}
	We proceed to control  $\|\nabla\nabla_{\y}f(\x_1,\bar{\y}_1)(\bz,\y-\hat{\y})\|$. To do so, we first estimate a closely related norm $\| \nabla\nabla_{\y}f(\x^{\theta},\y)(\bz,\y-\hat{\y})\|$. We apply Lemma \ref{le:hession} again to obtain
	\[\left\|\nabla_{\y} f\left(\x^{\theta},\hat{\y}\right)-\nabla_{\y}f\left(\x^{\theta},\y\right)-\nabla\nabla_{\y}f\left(\x^{\theta},\y\right)\left(\bz,\hat{\y}-\y\right) \right\|\leq\frac{H_f}{2}\left\|\hat{\y}-\y\right\|^2.\]
	Note that both $\y$ and $\hat{\y}$ belong to the set $\cS(\x^{\theta})$, which leads to  \[\nabla_{\y} f\left(\x^{\theta},\hat{\y}\right)=\nabla_{\y}f\left(\x^{\theta},\y\right)=\bz.\] 
	It follows that
	\begin{equation}\label{eq:nabla}
		\left\| \nabla\nabla_{\y}f\left(\x^{\theta},\y\right)\left(\bz,\y-\hat{\y}\right)\right\|\leq\frac{H_f}{2}\left\|\y-\hat{\y}\right\|^2.
	\end{equation}
    Putting everything together yields
	\begin{align*}
		&\left\|\nabla_{\y} f(\x_1,\hat{\y}_1)\right\|\\
		\leq& \left\|\left(\nabla\nabla_{\y}f(\x_1,\bar{\y}_1)-\nabla\nabla_{\y}f\left(\x^{\theta},\y\right)\right)(\bz,\y-\hat{\y})\right\|+\left\| \nabla\nabla_{\y}f\left(\x^{\theta},\y\right)\left(\bz,\y-\hat{\y}\right)\right\|\\
		&+	\frac{H_f}{2}\|\y-\hat{\y}\|^2 \\
		\leq&\left\|\nabla\nabla_{\y}f(\x_1,\bar{\y}_1)-\nabla\nabla_{\y}f\left(\x^{\theta},\y\right)\right\|\left\|(\bz,\y-\hat{\y})\right\|+	\frac{H_f}{2}\|\y-\hat{\y}\|^2  +\frac{H_f}{2}\|\y-\hat{\y}\|^2 \\
		\leq& H_f\left(\left\|\x_1-\x^{\theta}\right\|+\|\bar{\y}_1-\y\|\right)\cdot\left\|\y-\hat{\y} \right\|+{H_f}\|\y-\hat{\y}\|^2\\
		\leq&  \frac{H_f}{2}\|\x_1-{\x}^{\theta}\|^2+\frac{H_f}{2}\|\bar{\y}_1-\y\|^2+H_f\|\y-\hat{\y}\|^2+ {H_f}\|\y-\hat{\y}\|^2\\
		=& \frac{H_f}{2}\|\x_1-{\x}^{\theta}\|^2+\frac{H_f}{2}\|\bar{\y}_1-\y\|^2+2H_f\|\y-\hat{\y}\|^2,
	\end{align*}
	where the second inequality uses the definition of matrix's $l_2$ norm and \eqref{eq:nabla}; the third inequality is due to the $H_f$-Lipschitz continuity of $\nabla\nabla_{\y} f$ and the triangle inequality; the forth inequality is due to the Cauchy inequality.
	
	Combining the above estimate with $\|\x_1-\x^{\theta}\|=(1-\theta)\|\x_1-\x_2\|$, \eqref{eq:y1_y} (with $i=1$), and the bound from \eqref{eq:barytheta}, namely $\|\y-\hat{\y}\|\le 2(1-\theta)M_{\cS}\|\x_1-\x_2\|$, we obtain

	\begin{equation}
		\left\|\nabla_{\y} f(\x_1,\hat{\y}_1)\right\|\leq  \frac{H_f\left(1+17M_{\cS}^2\right)}2 (1-\theta)^2\|\x_1-\x_2\|^2.
	\end{equation}
	Finally, armed with the error bound condition in Assumption \ref{assum:basic} (A3), we have
\[\dist\left(\hat{\y}_1,\cS(\x_1)\right)\leq\frac{H_f\tau\left(1+17M_{\cS}^2\right)}2  (1-\theta)^2\|\x_1-\x_2\|^2. \]
	
	By the symmetric arguments, we can also have $\dist\left(\hat{\y}_2,\cS(\x_2)\right)\leq{ H_f\tau\left(1+17M_{\cS}^2\right)} \theta^2\|\x_1-\x_2\|^2$/2. We finished our proof for Claim \ref{claim:barytheta}.
        \end{proof}
	
	\subsection{Step 3: Define approximation points} 
	With the preparations in Steps~1–2, we now define the approximation points
\[
\y_1 \coloneqq \Pi_{\cS(\x_1)}(\hat{\y}_1), \qquad
\y_2 \coloneqq \Pi_{\cS(\x_2)}(\hat{\y}_2).
\]
To establish the smoothness of $\cS$, it remains to show that there exists a constant $L_{\cS}>0$ such that
	\begin{equation}\label{eq:desire1}
		\|\theta\y_1+(1-\theta)\y_2-\y\|\leq\frac{L_\cS}2\theta(1-\theta)\|\x_1-\x_2\|^2;
	\end{equation}  
	\begin{equation}\label{eq:desire2}
		\|\y_1-\y_2\|^2\leq  {L_\cS}\|\x_1-\x_2\|^2.
	\end{equation}
	We first prove the inequality \eqref{eq:desire1}.
We have 
	\begin{equation*}
		\begin{aligned}
			&\|\theta\y_1+(1-\theta)\y_2-\y\|\\
            = & \|\theta(\y_1-\hat\y_1)+(1-\theta)(\y_2-\hat\y_2) +(\theta \hat \y_1+(1-\theta)\hat\y_2-\y)\|\\
            \leq & \theta \|\y_1-\hat\y_1\| + (1-\theta)\|\y_1-\hat\y_1\|  +\|(\theta \hat \y_1+(1-\theta)\hat\y_2-\y)\|\\
			=&\theta \dist(\hat{\y}_1,\cS(\x_1))+(1-\theta)\dist(\hat{\y}_2,\cS(\x_2)) +\left\|\bar{\y}^{\theta}-\hat{\y}\right\|\\
=&\theta \dist(\hat{\y}_1,\cS(\x_1))+(1-\theta)\dist(\hat{\y}_2,\cS(\x_2))+\dist\left(\bar{\y}^{\theta},\cS\left(\x^{\theta}\right)\right),
		\end{aligned}
	\end{equation*}
    where the second equality follows from ${\y}_i=\Pi_{\cS(\x_i)}(\hat{\y}_i)$ and $\hat{\y}_i=\bar{\y}_i+\y-\hat{\y}$  for $i=1,2$; the final one is due to $\hat{\y}=\Pi_{\cS(\x^{\theta})}(\bar{\y}^{\theta})$.

	This, together with \eqref{eq:L0} and \eqref{eq:distbary}, implies
	\[	\|\theta\y_1+(1-\theta)\y_2-\y\|\leq(L_0+L_1)\theta(1-\theta)\|\x_1-\x_2\|^2.\]
	Hence, \eqref{eq:desire1} holds if ${L_\cS}\geq 2(L_0+L_1)$. It is left to show \eqref{eq:desire2}. 
	
	To prove \eqref{eq:desire2}, we estimate the distance $\|\y_1-\y_2\|$. To begin, we apply the triangle inequality to obtain
	\[ 	\|\y_1-\y_2\|\leq 	\|\y_1-\bar{\y}_1\|+\|\bar{\y}_1-\bar{\y}_2\|+\|\bar{\y}_2-\y_2 \|.\]
	Then, use the definition ${\y}_i=\Pi_{\cS(\x_i)}(\hat{\y}_i)=\Pi_{\cS(\x_i)}(\bar{\y}_i+\y-\hat{\y})$ and notice $\bar{\y}_i\in\cS(\x_i)$  for $i=1,2$. We have
	\begin{equation}\label{eq:tildey}
		\begin{aligned}
			\|\y_1-\y_2\|
			&\leq\left\|\Pi_{\cS(\x_1)}(\bar{\y}_1+(\y-\hat{\y}))-\Pi_{\cS(\x_1)}(\bar{\y}_1))\right\|+\|\bar{\y}_1-\bar{\y}_2\|\\
			&\quad\ +\left\|\Pi_{\cS(\x_2)}(\bar{\y}_2))-\Pi_{\cS(\x_2)}(\bar{\y}_2+(\y-\hat{\y}))\right\|\\
			&\leq \left\|\hat{\y}-\y\right\|+\|\bar{\y}_1-\bar{\y}_2\|+\left\|\hat{\y}-\y\right\|\\
			&=2\left\|\hat{\y}-\y\right\|+\|\bar{\y}_1-\bar{\y}_2\|,
		\end{aligned}
	\end{equation}
	where the second inequality follows from the non-expansiveness of the projectors $\Pi_{\cS(\x_1)}(\cdot)$ and $\Pi_{\cS(\x_2)}(\cdot)$
	
Recall from \eqref{eq:y1y2} that $\|\bar{\y}_1-\bar{\y}_2\|\le M_{\cS}\|\x_1-\x_2\|$, and note that \eqref{eq:barytheta} further implies
	\begin{equation}\label{eq:hatyy}
		\left\|\hat{\y}-\y\right\|\leq 2\theta(1-\theta)M_{\cS}\|\x_1-\x_2\|\leq\frac12M_{\cS}\|\x_1-\x_2\|.
	\end{equation}
	Combining \eqref{eq:tildey} and \eqref{eq:hatyy} yields  
	\[\|\y_1-\y_2\|\leq2M_{\cS}\|\x_1-\x_2\|, \]
	which proves \eqref{eq:desire2} with ${L_\cS}\geq 4M_{\cS}^2$. 
	
Finally, to ensure \eqref{eq:desire1} it suffices to choose
\[
L_{\cS}\;\ge\;\max\{\,2(L_0+L_1),\,4M_{\cS}^2\,\}.
\]
Recalling $M_{\cS}=L_f\tau$, $L_0=\tfrac{H_f\tau}{2}(1+M_{\cS}^2)$, and
$L_1=\tfrac{H_f\tau}{2}(1+17M_{\cS}^2)$, we obtain
\[
L_{\cS}
=\max\Big\{\,2H_f\tau\big(1+9L_f^2\tau^2\big),\;4L_f^2\tau^2\,\Big\}.
\]
This completes the proof of Theorem \ref{pro:Swc}.

	\subsection{Proof of Remark \ref{re:Pro2bounded}}
	In the proof of Theorem \ref{pro:Swc}, Assumption \ref{assum:basic} is only used to guarantee Lemma \ref{le:Lipschitz} and the estimations on the involved points. Notice that to ensure Lemma \ref{le:Lipschitz}, it suffices for Assumption \ref{assum:basic} to hold on the set $\cD\times{\rm Conv}(\bigcup_{\x\in\cD}\cS(\x))\subseteq\cD\times\cY$. We only need to check that the involved points belong to the set $\cD\times\cY$ to prove Remark \ref{re:Pro2bounded}.
	
	As the domain of $\cS$ is $\cD$, to verify the set smoothness of $\cS$, we can choose $\x_1,\x_2\in\cD$. Then, we have $\x^{\theta}\in\cD$ due to the convexity of $\cD$. Hence, it suffices to check that $\bar\y_1$, $\bar\y_2$, $\bar{\y}^{\theta}$,  $\y_1$, $\y_2$, $\hat{\y}$, $\hat{\y}_1$, and $\hat\y_2$ belong to the convex set $\cY$. 
	
	First, recall that $\bar\y_1,\y_1\in\cS(\x_1),\bar\y_2,\y_2\in\cS(\x_2)$, $\hat{\y}\in\cS(\x^{\theta})$, and $\bar{\y}^{\theta}\in {\rm Conv}(\cS(\x_1)\cup\cS(\x_2))\subseteq {\rm Conv}(\bigcup_{\x\in\cD}\cS(\x))$ due to their definitions. We see that they belong to the set $\cY$. We then focus on $\hat{\y}_1$ and $\hat\y_2$. Recall that $\hat\y_i=\bar{\y}_i+(\y-\hat{\y})$ for $i=1,2$. We see that for $i=1,2$,
	\[\dist\left(\hat\y_i,\cS(\x_i)\right)\leq\|\hat\y_i-\bar{\y}_i \|=\left\|\y-\hat{\y}\right\|. \]
	By \eqref{eq:hatyy} and $\|\x_1-\x_2\|\leq{\rm diam}(\cD)$, for $i=1,2$, we further have
	\[ \dist\left(\hat\y_i,\cS(\x_i)\right)\leq\frac12M_{\cS}\cdot{\rm diam}(\cD).\]
	This implies that $\hat{\y}_i\in\cS(\x_i)+\frac12M_{\cS}\cdot{\rm diam}(\cD)\cdot\B(\bz,1)\subseteq\cY$ for $i=1,2$. We complete the proof.
	
	\section{Proof of Theorem \ref{th:compute} }
	We first introduce some background on zeroth-order methods before the formal proof.
	Let $\P$ denote the uniform distribution on the unit sphere in $\R^m$. Given a function $g:\R^m\to\R$ and a radius $\varepsilon>0$, we define the randomized smooth approximation $g$ by $g^{\varepsilon}(\x)\coloneqq \E_{\u\sim\P}[g(\x+\varepsilon\u)]$. We have the following properties for $g^{\varepsilon}$. 
	\vspace{1em}
	\begin{lemma}[Basic Properties of Randomized Smoothing]\label{le:gaussian} The following hold:
		
		\begin{enumerate}[label={{\rm (\roman*)}}]
			\item If $g$ is $M_g$-Lipschitz continuous, then $g^{\varepsilon}$ is differentiable, $M_g$-Lipschitz continuous, and satisfies
			\begin{equation}\label{eq:g_approx}
				\left|g(\x)-g^{\varepsilon}(\x)\right|\leq \varepsilon M_g.
			\end{equation}
			\item If $g$ is $r$-weakly convex (resp. concave), then $g^{\varepsilon}$ is $r$-weakly convex (resp. concave).
		\end{enumerate}
	\end{lemma}
	\begin{proof}
		(i) See \cite[Proposition 2.3]{lin2022gradient}.
		
		(ii) If $g$ is $r$-weakly convex, then by the same arguments of \citet[Lemma 16]{nazari2020adaptive}, we have the $r$-weak convexity of $g^{\varepsilon}$. When $g$ is $r$-weakly concave, we note that $-g$ is $r$-weakly convex, and hence $(-g)^{\varepsilon}$ is $r$-weakly convex. Due to the simple fact $(-g)^{\varepsilon}=-g^{\varepsilon}$, we have the $r$-weak convexity of $-g^{\varepsilon}$, i.e., the $r$-weak concavity of $g^{\varepsilon}$.
	\end{proof}
	As the approximation for a subdifferental $\partial g$ via randomized smoothing can be inexact, we need the following $\epsilon$-subdifferential.
	\vspace{1em}
	\begin{definition}\label{def:epsilon_sub}
		Consider a convex function $g:\R^m\to\R$ and a scalar $\nu\geq0$. We define the $\nu$-subdifferential of $g$ at $\x\in\R^m$ by
		\[\partial_{\nu} g(\x)=\left\{\s\in\R^m:g(\z)\geq g(\x)+\s^T(\z-\x)-\nu,\ \forall\ \z\in\R^m\right\}.\]
	\end{definition}
	We then develop the convergence rates of IZOM for optimistic and pessimistic BLO, respectively. We remark that our analysis remains valid when the hyper-objective functions $\varphi_o$ (resp. $\varphi_p$) are replaced with general $M$-Lipschitz continuous, $\rho$-weakly concave (resp. convex) functions.

	\subsection{Optimistic Case}\label{subsec:optimistic}
	To begin, we record a celebrated proposition on subdifferential transportation.
	\vspace{1em}
	\begin{proposition}\label{pro:trans}(cf. \cite[Theorem 5.5]{atenas2023unified} and \cite[Theorem 2]{robinson1999linear})
		Let $g:\R^m\to\R$ be a proper lower semicontinuous convex function. Suppose that $\nu\geq0$ and $G\in\partial_{\nu}g(\x)$. Then, for each $r>0$, there is a unique vector $\v\in\R^m$ such that 
		\[G-\frac{1}{r}\v\in\partial g(\x+r\v),\  \|\v\|\leq \sqrt{\nu}.\]
	\end{proposition}
	Proposition \ref{pro:trans} plays an important role in relating an $\epsilon$-subdifferential to the Clarke subdifferential at a near point, leading to the following lemma.
	\begin{lemma}\label{le:trans}
		Let $g:\R^m\to\R$ be an  $M_g$-Lipschitz continuous and $\rho$-weakly concave function. Let $g^{\varepsilon}$ be the randomized approximation of $g$ with radius $\varepsilon>0$ and $\nu=2\varepsilon M_g$. Then,	for all $\x\in\R^m$, we have
		\begin{equation}\label{eq:phi_delta}
			\dist\left(\bz,{\textstyle \bigcup}_{\z\in\B(\x,\sqrt{\nu})}\partial g(\z)\right)\leq \left\|\nabla g^{\varepsilon}(\x)\right\|+(\rho+1)\sqrt{\nu}.
		\end{equation}
	\end{lemma}
	\begin{proof}[Proof of Lemma \ref{le:trans}]
		By Lemma \ref{le:gaussian} (ii), $g^{\varepsilon}$ is $\rho$-weakly concave, i.e., the function $\x\mapsto\frac{\rho}{2}\|\x\|^2-g^{\varepsilon}(\x)$ is convex. Then, we know that for all $\z,\x\in\R^m$,
		\begin{equation*}
			\frac{\rho}{2}\|\z\|^2-g^{\varepsilon}(\z)\geq \frac{\rho}{2}\|\x\|^2-g^{\varepsilon}(\x)+\left(\rho\x-\nabla g^{\varepsilon}(\x)\right)^T(\z-\x).
		\end{equation*}
		This, together with Lemma \ref{le:gaussian} (i), implies
		\begin{equation*}
			\frac{\rho}{2}\|\z\|^2-g(\z)\geq \frac{\rho}{2}\|\x\|^2-g(\x)+\left(\rho\x-\nabla g^{\varepsilon}(\x)\right)^T(\z-\x)-2\varepsilon M_g,
		\end{equation*}
		which is equivalent to
		\begin{equation*}
			\rho\x-\nabla g^{\varepsilon}(\x)\in \partial_{\nu}\left(\frac{\rho}{2}\|\x\|^2-g(\x)\right)\quad\text{ with }\quad\nu=2\varepsilon M_{g}.
		\end{equation*}
		For simplicity, we define $\bar{g}:\x\mapsto\frac{\rho}{2}\|\x\|^2-g(\x)$.
		Clearly, $\bar{g}$ is convex due to the $\rho$-weak concavity of $g$. Applying Proposition \ref{pro:trans}, we see that there exists $\v\in\R^m$ with $\|\v\|\leq \sqrt{\nu}$ such that 
		\begin{equation}\label{eq:barphi}
			\rho\x-\nabla g^{\varepsilon}(\x)-\v\in\partial \bar{g}(\x+\v).
		\end{equation}
		Recall that $-g$ is $\rho$-weakly convex, and thus is regular according to \cite[Proposition 4.5]{vial1983strong}. Then, we have $\partial \bar{g}(\z)=\rho\z+\partial(-g)(\z)$ for all $\z\in\R^m$ by \cite[Corollary 3 of Proposition 2.3.3]{clarke1990optimization}. On the other hand, we have $\partial (-g)=-\partial g$ by \cite[Proposition 2.3.1]{clarke1990optimization} and Lipschitz continuity of $g$. Hence, we see that 
		\[\partial \bar{g}(\z)=\rho\z-\partial g(\z),\qquad \forall\ \z\in\R^m. \]
		In particular, we have
		\[\partial \bar{g}(\x+\v)= \rho(\x+\v)-\partial g(\x+\v).\]
		This, together with \eqref{eq:barphi}, implies
		\begin{equation*}
			\nabla g^{\varepsilon}(\x)+(\rho+1)\v\in \partial g(\x+\v).
		\end{equation*}
		It follows that 
		\[\dist(\bz,\partial g(\x+\v))\leq \left\|\nabla g^{\varepsilon}(\x)+(\rho+1)\v \right\|\leq \left\|\nabla g^{\varepsilon}(\x)\right\|+(\rho+1)\|\v\|.\]
		Observe that $\|\v\|\leq\sqrt{\nu}$ and  
		\[\dist\left(\bz,\partial g(\x+\v)\right)\geq \dist\left(\bz,{\textstyle \bigcup}_{\z\in\B(\x,\|\v\|)}\partial g(\z)\right)\geq \dist\left(\bz,{\textstyle \bigcup}_{\z\in\B(\x,\sqrt{\nu})}\partial g(\z)\right).\]
		We obtain the desired inequality
		\[\dist\left(\bz,{\textstyle \bigcup}_{\z\in\B(\x,\sqrt{\nu})}\partial g(\z)\right)\leq \left\|\nabla g^{\varepsilon}(\x)\right\|+(\rho+1)\sqrt{\nu}.\]
	\end{proof}
	
	Now, we are ready to prove the convergence rate for IZOM. To begin, we define the subdifferential approximation function $G$ by
	\[G(\x_t)=\frac{m}{2\varepsilon}(\varphi_o(\x_t+\varepsilon\u_t)-\varphi_o(\x_t-\varepsilon\u_t) )\u_t.\]
	By \cite[Lemma D.1]{lin2022gradient}, it holds that
	\begin{equation}\label{eq:G_1}
		\E[G(\x_t)|\x_t]=\nabla  \varphi_o^{\varepsilon}(\x_t);\quad \E[\|G(\x_t)\|^2|\x_t]\leq 16\sqrt{2\pi}m M_{\varphi}^2.
	\end{equation}
	Since $|\tilde{\varphi}(\x_t)-\varphi_o(\x_t)|\leq w$ by the subroutine $\cA$, we have $\|\tilde{G}(\x_t)-{G}(\x_t)\|\leq \frac{mw}{\varepsilon}$. This, together with the simple fact $\|\tilde{G}(\x_t)\|^2\leq 2\|\tilde{G}(\x_t)-{G}(\x_t)\|^2+2\|G(\x_t)\|^2$ due to the Cauchy inequality, implies
	\begin{equation}\label{eq:G_2}
		\E[\|\tilde{G}(\x_t)-{G}(\x_t)\|^2|\x_t]\leq \left(\frac{mw}{\varepsilon}\right)^2;\quad \E[\|\tilde{G}(\x_t)\|^2|\x_t]\leq 32\sqrt{2\pi}m M_{\varphi}^2+2\left(\frac{mw}{\varepsilon}\right)^2.
	\end{equation}
	Next, we combine the update of IZOM and $\rho$-weak concavity of $\varphi^{\varepsilon}_o$
	to develop a sufficient decrease property for the $t$-th iteration. Using Fact \ref{fact:relative} and the update $\x_{t+1}-\x_t=-\eta \tilde{G}(\x_t)$ of Algorithm \ref{al:1}, we obtain the following estimate:
	\[\begin{aligned}
		&\varphi_o^{\varepsilon}(\x_{t+1})\\
		\leq& \varphi_o^{\varepsilon}(\x_{t})+\nabla \varphi_o^{\varepsilon}(\x_{t})^T(\x_{t+1}-\x_t)+\frac{\rho}{2}\|\x_{t+1}-\x_t\|^2\\
		=& \varphi_o^{\varepsilon}(\x_{t})-\eta\nabla \varphi_o^{\varepsilon}(\x_t)^T\tilde{G}(\x_t)+\frac{\rho}{2}\eta^2 \|\tilde{G}(\x_t)\|^2\\
		=&\varphi_o^{\varepsilon}(\x_{t})-\eta\nabla \varphi_o^{\varepsilon}(\x_t)^T{G}(\x_t)-\eta\nabla \varphi_o^{\varepsilon}(\x_t)^T\left(\tilde{G}(\x_t)-{G}(\x_t)\right)+\frac{\rho}{2}\eta^2 \|\tilde{G}(\x_t)\|^2\\
		\leq& \varphi_o^{\varepsilon}(\x_{t})-\eta\nabla \varphi_o^{\varepsilon}(\x_t)^T{G}(\x_t)+\frac{\eta}2\|\nabla \varphi_o^{\varepsilon}(\x_t)\|^2+\frac{\eta}2\left\|\tilde{G}(\x_t)-{G}(\x_t)\right\|^2+\frac{\rho}{2}\eta^2 \|\tilde{G}(\x_t)\|^2.\\
	\end{aligned}
	\]
	where the last inequality is due to the Cauchy inequality. 
	
	We take expectation conditioning on $\x_t$ for the above inequality. Recall that $\E[{G}(\x_t)|\x_t]=\nabla \varphi_o^{\varepsilon}(\x_t)$ by \eqref{eq:G_1}. We see that
	\begin{equation*}
		\begin{aligned}
			\E[\varphi_o^{\varepsilon}(\x_{t+1})|\x_t]&\leq \varphi_o^{\varepsilon}(\x_{t})-\frac{\eta}2\|\nabla \varphi_o^{\varepsilon}(\x_t)\|^2+\frac{\eta}2\E[\|\tilde{G}(\x_t)-{G}(\x_t)\|^2|\x_t]  \\
			&\quad +\frac{\rho\eta^2}{2}\E[\|\tilde{G}(\x_t)\|^2|\x_t].
		\end{aligned}
	\end{equation*} 
	Apply Lemma \ref{le:trans} to $\varphi_o$ and use the Cauchy inequality. We obtain
	\[\dist\left(\bz,{\textstyle \bigcup}_{\z\in\B\left(\x_t,\sqrt{\nu}\right)}\partial \varphi_o(\z)\right)^2\leq 2\left\|\nabla\varphi_o^{\varepsilon}(\x_t)\right\|^2+2(\rho+1)^2\varepsilon. \]
	Combining the above two inequalities with \eqref{eq:G_2}, we have
	\begin{equation*}
		\begin{aligned}
			\E[\varphi_o^{\varepsilon}(\x_{t+1})|\x_t]&\leq \varphi_o^{\varepsilon}(\x_{t})-\frac{\eta}4\dist\left(\bz,{\textstyle \bigcup}_{\z\in\B\left(\x_t,\sqrt{\nu}\right)}\partial \varphi_o(\z)\right)^2+\frac{\eta}{2}(\rho+1)^2\varepsilon\\
			&\quad+\frac{\eta}2\left(\frac{mw}{\varepsilon}\right)^2+\frac{\rho}{2}\eta^2 \left( 32\sqrt{2\pi}m M_{\varphi}^2+2\left(\frac{mw}{\varepsilon}\right)^2\right).
		\end{aligned}
	\end{equation*}
	Recall that $\nu=2\varepsilon M_{\varphi}$, $\eta=\Theta(\frac1{\sqrt{mT}}), \varepsilon=\cO(\frac1{\sqrt{T}})$, and $ w=\cO(\frac1{m^{\frac34}T^{\frac34}})$. Ignoring some scalars, we further have
	\begin{equation} \label{eq:2T}
		\E[\varphi_o^{\varepsilon}(\x_{t+1})|\x_t]\leq \varphi_o^{\varepsilon}(\x_{t})-\frac{\eta}4\dist\left(\bz,{\textstyle \bigcup}_{\z\in\B\left(\x_t,\sqrt{\nu}\right)}\partial \varphi_o(\z)\right)^2+\cO\left(\frac1T\right).
	\end{equation}
	Summing \eqref{eq:2T} over $t=0,1,\ldots T-1$ and taking full expectation, we obtain
	\[\E[\varphi_o^{\varepsilon}(\x_{T})]\leq\varphi_o^{\varepsilon}(\x_{0})-\frac{\eta}4\sum_{t=0}^{T-1}\E\left[\dist\left(\bz,{\textstyle \bigcup}_{\z\in\B\left(\x_t,\sqrt{\nu}\right)}\partial \varphi_o(\z)\right)^2\right]+\cO(1). \]
	Note that the definition of $\bar\x$ yields 
    \[\E\left[\dist\left(\bz,{\textstyle \bigcup}_{\z\in\B\left(\bar\x,\sqrt{\nu}\right)}\partial \varphi_o(\z)\right)^2\right]=\frac1T\sum_{t=0}^{T-1}\E\left[\dist\left(\bz,{\textstyle \bigcup}_{\z\in\B(\x_t,\sqrt{\nu})}\partial \varphi_o(\z)\right)^2\right].\] 
	On the other hand, by Lemma \ref{le:gaussian} (i) and $M_{\varphi}$-Lipschitz continuity of $\varphi_o$, 
	\[\varphi_o^{\varepsilon}(\x_{0})-\E[\varphi_o^{\varepsilon}(\x_{T})]\leq 2\varepsilon M_{\varphi}+\varphi_o(\x_0)-\min_{\x}\varphi_o(\x)=\Delta_o.\]
	
	Putting all the things together, we have
	\[\E\left[\dist\left(\bz,{\textstyle \bigcup}_{\z\in\B\left(\bar\x,\sqrt{\nu}\right)}\partial \varphi_o(\z)\right)^2\right]=\cO\left(\frac{\Delta_o+1}{\eta T}\right)=\cO\left(\frac{\sqrt{m}(\Delta_o+1)}{\sqrt{T}}\right), \]
	where $\sqrt{\nu}=\sqrt{2\varepsilon M_{\varphi}}=\cO(T^{-\frac14})$. 
	
	\subsection{Pessimistic Case }
	In this section, we let  $G(\x_t)=\frac{m}{2\varepsilon}(\varphi_p(\x_t+\varepsilon\u_t)-\varphi_p(\x_t-\varepsilon\u_t) )\u_t$ and use $\hat{\x}$ to denote $\prox_{\gamma,\varphi_p}(\x)$ for simplicity. Similar to the arguments on \eqref{eq:G_1}, \eqref{eq:G_2}, we have the following due to \cite[Lemma D.1]{lin2022gradient} and $|\tilde{\varphi}(\x_t)-\varphi_p(\x_t)|\leq w$ given by Algorithm \ref{al:1}:
	\begin{equation}\label{eq:G1}
		\E[G(\x_t)|\x_t]=\nabla  \varphi_p^{\varepsilon}(\x_t);\quad \E[\|G(\x_t)\|^2|\x_t]\leq 16\sqrt{2\pi}m M_{\varphi}^2.
	\end{equation}
	\begin{equation}\label{eq:G2}
		\E[\|\tilde{G}(\x_t)-{G}(\x_t)\|^2|\x_t]\leq \left(\frac{mw}{\varepsilon}\right)^2;\quad \E[\|\tilde{G}(\x_t)\|^2|\x_t]\leq 32\sqrt{2\pi}m M_{\varphi}^2+2\left(\frac{mw}{\varepsilon}\right)^2.
	\end{equation}
	Invoking the methodology of \citet{davis2019stochastic},
	we first estimate $\|\hat{\x}_t-\x_{t+1}\|$ for $t=0,1,\ldots,T-1$. To begin, the update of Algorithm \ref{al:1} and direct computation give the following estimate:
	\begin{equation}\label{eq:hatx}
		\begin{aligned}
			&\|\hat{\x}_t-\x_{t+1}\|^2\\
			=&\|\hat{\x}_t-\x_t+\eta \tilde{G}(\x_t)\| \\
			=&\|\hat{\x}_t-\x_t\|^2+\eta^2\|\tilde{G}(\x_t)\|^2+2\eta\tilde{G}(\x_t)^T(\hat{\x}_t-\x_t)\\		
			=& \|\hat{\x}_t-\x_t\|^2+\eta^2\|\tilde{G}(\x_t)\|^2+2\eta{G}(\x_t)^T(\hat{\x}_t-\x_t)+2\eta\left(\tilde{G}(\x_t)-{G}(\x_t)\right)^T(\hat{\x}_t-\x_t)\\
			\leq& \|\hat{\x}_t-\x_t\|^2+\eta^2\|\tilde{G}(\x_t)\|^2+2\eta{G}(\x_t)^T(\hat{\x}_t-\x_t)+\eta\|\tilde{G}(\x_t)-{G}(\x_t)\|^2+\eta \|\hat{\x}_t-\x_t\|^2.\\
		\end{aligned}
	\end{equation}
	Taking expectation in \eqref{eq:hatx} and using \eqref{eq:G1}, \eqref{eq:G2}, we obtain the following inequality:
	\begin{equation}\label{eq:hatx1}
		\begin{aligned}
			&\E[\|\hat{\x}_t-\x_{t+1}\|^2|\x_t]\\
			\leq &(1+\eta)\|\hat{\x}_t-\x_{t}\|^2+2\eta\nabla  \varphi_p^{\varepsilon}(\x_t)^T(\hat{\x}_t-\x_t)+\eta 	\E[\|\tilde{G}(\x_t)-{G}(\x_t)\|^2|\x_t]\\
			&\ +\eta^2\E[\|\tilde{G}(\x_t)\|^2|\x_t]\\
			\leq &(1+\eta)\|\hat{\x}_t-\x_{t}\|^2+2\eta\nabla  \varphi_p^{\varepsilon}(\x_t)^T(\hat{\x}_t-\x_t) +\eta\left(\frac{mw}{\varepsilon}\right)^2+32\sqrt{2\pi}m M_{\varphi}^2\eta^2+2\left(\frac{mw}{\varepsilon}\right)^2\eta^2.
		\end{aligned}
	\end{equation}
	We then turn to estimate $\varphi_{p,\gamma}(\x_{t+1})$. By the definition of $\varphi_{p,\gamma}$, we have \[\varphi_{p,\gamma}(\x_{t+1})\leq \varphi_p(\hat{\x}_t)+\frac{1}{2\gamma}\|\hat{\x}_t-\x_{t+1}\|^2.\]
	This, together with \eqref{eq:hatx1}, implies
	\begin{equation*}
		\begin{aligned}
			&\E[\varphi_{p,\gamma}(\x_{t+1})|\x_t]-\varphi_{p}(\hat{\x}_t)\\
			\leq&  \frac{1+\eta}{2\gamma}\|\hat{\x}_t-\x_{t}\|^2+\frac{\eta}{\gamma}\nabla  \varphi_p^{\varepsilon}(\x_t)^T(\hat{\x}_t-\x_t) +\frac{\eta}{2\gamma}\left(\frac{mw}{\varepsilon}\right)^2+16\sqrt{2\pi}m M_{\varphi}^2\frac{\eta^2}{\gamma}+\left(\frac{mw}{\varepsilon}\right)^2\frac{\eta^2}{\gamma}.
		\end{aligned}
	\end{equation*}
	Recall that $\varphi_p^{\varepsilon}$ is $\rho$-weakly convex by Theorem \ref{th:main} and Lemma \ref{le:gaussian} (ii). Then, using Lemma \ref{le:wc} (iii), we have
	\[\begin{aligned}
		\nabla  \varphi_p^{\varepsilon}(\x_t)^T(\hat{\x}_t-\x_t)&\leq \varphi_p^{\varepsilon}(\hat{\x}_t)-\varphi_p^{\varepsilon}(\x_t)+\frac{\rho}{2}\|\hat{\x}_t-\x_t\|^2\\
		&\leq \varphi_p(\hat{\x}_t)-\varphi_p(\x_t)+\frac{\rho}{2}\|\hat{\x}_t-\x_t\|^2+2\varepsilon M_{\varphi},
	\end{aligned}
	\]
	where the second inequality is due to Lemma \ref{le:gaussian} (i).
	
	Combining the above two estimates gives
	\begin{equation*}
		\begin{aligned}
			&\E[\varphi_{p,\gamma}(\x_{t+1})|\x_t]\\
			\leq &\varphi_{p}(\hat{\x}_t)+\frac{1}{2\gamma}\|\hat{\x}_t-\x_{t}\|^2+\frac{\eta}{\gamma}\left( \varphi_p(\hat{\x}_t)-\varphi_p(\x_t)+\frac{\rho+1}{2}\|\hat{\x}_t-\x_t\|^2\right)+\frac{2\eta}{\gamma}\varepsilon M_{\varphi}\\ 
			&\ +\frac{\eta}{2\gamma}\left(\frac{mw}{\varepsilon}\right)^2+16\sqrt{2\pi}m M_{\varphi}^2\frac{\eta^2}{\gamma}+\left(\frac{mw}{\varepsilon}\right)^2\frac{\eta^2}{\gamma} \\
			= &\varphi_{p,\gamma}(\x_{t})+ \frac{\eta}{\gamma}\left(\varphi_{p,\gamma}(\x_t)-\varphi_p(\x_t)+\frac{(\rho+1)\gamma-1}{2\gamma}\|\hat{\x}_t-\x_t\|^2\right)+\frac{2\eta}{\gamma}\varepsilon M_{\varphi}\\ 
			&\ +\frac{\eta}{2\gamma}\left(\frac{mw}{\varepsilon}\right)^2+16\sqrt{2\pi}m M_{\varphi}^2\frac{\eta^2}{\gamma}+\left(\frac{mw}{\varepsilon}\right)^2\frac{\eta^2}{\gamma}, \\
		\end{aligned}
	\end{equation*}
	where the equation uses $\varphi_{p,\gamma}(\x_t)=\varphi_p(\hat{\x}_t)+\frac{1}{2\gamma}\|\hat{\x}_t-\x_{t}\|^2$.
	Recall that by Lemma \ref{le:moreau}, it holds that $\x_t-\hat{\x}_t=\gamma\nabla \varphi_{p,\gamma}(\x_t)$ and \[\varphi_{p,\gamma}(\x_t)-\varphi_p(\x_t)\leq -\frac{1-\gamma\rho}{2\gamma}\|\hat{\x}_t-\x_{t}\|^2\leq -\frac{1-\gamma(\rho+1)}{2\gamma}\|\hat{\x}_t-\x_{t}\|^2.\] 
	We further have
	\begin{equation*}
		\begin{aligned}
			\E[\varphi_{p,\gamma}(\x_{t+1})|\x_t]&\leq\varphi_{p,\gamma}(\x_t)-\eta\left(1-\gamma(\rho+1)\right)\|\nabla \varphi_{p,\gamma}(\x_t)\|^2+\frac{2\eta}{\gamma}\varepsilon M_{\varphi}\\		
			&\quad +\frac{\eta}{2\gamma}\left(\frac{mw}{\varepsilon}\right)^2+16\sqrt{2\pi}m M_{\varphi}^2\frac{\eta^2}{\gamma}+\left(\frac{mw}{\varepsilon}\right)^2\frac{\eta^2}{\gamma}.
		\end{aligned}
	\end{equation*}
	Recall that $\eta=\Theta(\frac1{\sqrt{mT}}), \varepsilon=\cO(\frac1{\sqrt{T}})$, $ w=\cO(\frac1{m^{\frac34}T^{\frac34}})$, and $\gamma\in(0,1/(\rho+1))$. Neglecting some scalars, it follows that
	\begin{equation}\label{eq:1T}
		\E[\varphi_{p,\gamma}(\x_{t+1})|\x_t]\leq \varphi_{p,\gamma}(\x_t)-\eta\left(1-\gamma(\rho+1)\right)\|\nabla \varphi_{p,\gamma}(\x_t)\|^2+\cO\left(\frac1T \right).
	\end{equation}
	Summing \eqref{eq:1T} over $t=0,1,\ldots T-1$ and taking full expectation, we have
	\[\E[\varphi_{p,\gamma}(\x_T)]\leq\varphi_{p,\gamma}(\x_0)-\eta\left(1-\gamma(\rho+1)\right)\sum_{t=0}^{T-1}\E[\|\nabla \varphi_{p,\gamma}(\x_t)\|^2]+\cO(1). \]
	Note that the definition of $\bar\x$ gives $\sum_{t=0}^{T-1}\|\nabla \varphi_{p,\gamma}(\x_t)\|^2=T\E[\|\nabla \varphi_{p,\gamma}(\bar\x)\|^2]$. Also, notice that
	\[\varphi_{p,\gamma}(\x_0)-\E[\varphi_{p,\gamma}(\x_T)]\leq\varphi_{p,\gamma}(\x_0)-\min_{\x} \varphi_{p,\gamma}(\x)=\Delta_p. \]
	Putting all the things together, we obtain
	\[\E[\|\nabla \varphi_{p,\gamma}(\bar\x)\|^2]=\cO\left({\frac{\Delta_p+1}{\eta{T}}}\right)=\cO\left({\frac{\sqrt{m}(\Delta_p+1)}{\sqrt{T}}}\right).\]
	This completes the proof.

\end{document}